\documentclass[reqno,10pt]{amsart}
\usepackage{amssymb,amsmath,amsthm,amstext,amsfonts}
\usepackage[dvipdfmx]{graphicx}
\usepackage{psfrag}
\usepackage{color}

\makeatletter \@addtoreset{equation}{section} \makeatother

\renewcommand\thetable{\thesection.\@arabic\c@table}

\theoremstyle{plain}
\newtheorem{maintheorem}{Theorem}
\newtheorem{theorem}{Theorem }[section]
\newtheorem{proposition}[theorem]{Proposition}

\newtheorem{lemma}[theorem]{Lemma}

\newtheorem{maincorollary}{Corollary}

\theoremstyle{definition} \theoremstyle{remark}
\newtheorem{remark}[theorem]{Remark}

\newcommand{\vep}{\varepsilon}

\newcommand{\cF}{\mathcal{F}}

\newcommand{\cA}{\mathcal{A}}

\theoremstyle{remark}
\newtheorem{defi}{\textup{Definition}}[section]

\newtheorem{rem}[defi]{\textup{Remark}}

\theoremstyle{definition}

\theoremstyle{plain}

\makeatletter
\@addtoreset{equation}{section}

\begin{document}

\title{Specification and partial hyperbolicity for flows}

 \begin{abstract} 
We prove that if a flow exhibits a partially hyperbolic attractor $\Lambda$ with splitting 
$T_\Lambda M=E^s\oplus E^c$  and two periodic saddles with different indices such that
the stable index of one of them coincides with the dimension of $E^s$ then it does not satisfy 
the specification property. In particular, every singular-hyperbolic attractor with the 
specification property is hyperbolic.
As an application, we prove that no Lorenz attractor satisfies the specification property.
 \end{abstract}

\author[N. Sumi]{Naoya Sumi}
\address{Department of Mathematics, Faculty of Science \\
Kumamoto University \\
2-39-1 Kurokami, Kumamoto-shi, Kumamoto, 860-8555, JAPAN
}
\email{sumi@sci.kumamoto-u.ac.jp}

\author[P. Varandas]{Paulo Varandas}
\address{Departamento de Matem\'atica, Universidade Federal da Bahia\\
Av. Ademar de Barros s/n, 40170-110 Salvador, Brazil  \\ \& 
CMUP - University of Porto, Portugal.}
\email{paulo.varandas@ufba.br }

\author[K. Yamamoto]{Kenichiro Yamamoto}
\address{Department of General Education \\
Nagaoka University of Technology \\
Niigata 940-2188, JAPAN
}
\email{k\_yamamoto@vos.nagaokaut.ac.jp}

\date{\today}

\maketitle

\section{Introduction}

The purpose of this paper is to give a characterization of $C^1$-flows on 
compact Riemannian manifolds that have attractors with the specification property.
The specification property for maps and flows was introduced by Bowen in \cite{Bo71,Bo72} 
and roughly means that an arbitrary number of pieces of orbits can be ``glued" to obtain a real orbit that 
shadows the previous ones.  The relevance in the study of this property is that it plays a key role e.g.
in the study of the uniqueness of equilibrium states (\cite{B}), large deviations theory (\cite{Y}) 
and multifractal analysis (\cite{TaV00,TaV03}).
Those are some of the reasons for which dynamical systems satisfying the specification property 
have been intensively studied from an ergodic viewpoint \cite{B,S, Va12, OT} and from an algebraic 
viewpoint \cite{ADK,L}.  This justifies the interest of many researchers to obtain weaker forms of 
specification (see e.g.~\cite{OT,PS05,Th10,Va12} and references therein).

Using that the specification property is well known to imply topologically mixing (see \cite{DGS}),
 a first conceptual difference appears when considering the specification property in the discrete
or the continuous time setting. Indeed while, up to consider a finite iterate, every uniformly hyperbolic diffeomorphism restricted to every basic piece satisfies the specification property (see~\cite{KH})
there are simple constructions of uniformly hyperbolic flows (e.g. obtained as suspension of a transitive 
Anosov diffeomorphism by a constant roof function) that are not even topologically mixing.

In the nineties, Palis proposed a conjecture for a 
global view of dynamics which has been a routing guide for many works in the last years, which we 
describe here in the space $C^1$-diffeomorphisms and flows: either the dynamics is 
uniformly hyperbolic or it can be $C^1$-approximated by one other that exhibits a homoclinic 
tangency or a heteroclinic cycle. In rough terms, in the complement of uniform hyperbolicity (open condition)
the mechanisms that generate non-hyperbolicity in a dense way are tangencies and cycles.
We refer the reader to the surveys \cite{Bonatti,Palis} for reports on the advances towards the conjecture
and the current state of the conjecture.

Palis and the $C^1$-stability conjectures (c.f. \cite{Hayashi,Mane, Wen,GW06}) inspired the works of many authors
to approach such dichotomy in the space of $C^1$-diffeomorphisms concerning other important dynamical properties  
that are not necessarily $C^1$-open, namely, expansiveness, shadowing or specification properties. 
In \cite{Arbietoexp,PT,S,SSY} it was proved that the
$C^1$-interior of the set of all $C^1$-diffeomorphisms satisfying any of 
these properties is contained in the set of uniformly hyperbolic diffeomorphisms.

Let us describe more carefully some results concerning the characterization of diffeomorphisms with the specification property. In \cite{SSY}, Sakai together with the first and third authors
proved that the $C^1$-interior of the set of all diffeomorphisms satisfying the specification property
coincides with the set of all transitive Anosov diffeomorphisms.
Moriyasu, Sakai and the third author extended the above results to
regular maps, and proved that $C^1$-generically, regular maps satisfy the specification property
if and only if they are transitive Anosov (\cite{MSY}).
In \cite{SVY} we proved that the presence of periodic points with different indexes is an obstruction for
specification even for partially hyperbolic diffeomorphisms.
Owing to these results, the relation between specification and hyperbolicity for $C^1$-diffeomorphisms
turns out to be clear.

The characterization of the smooth flows with the specification property in comparison to the discrete time 
setting presents both conceptual and technical difficulties. The fact that critical elements for flows include 
not only periodic orbits as singularities constitutes an obstacle to follow the same lines of the argument 
in~\cite{SSY}.
Arbieto, Senos and Todero~\cite{AST} were able to overcome these difficulties and
proved that if a flow $(X_t)_{t\in \mathbb R}$ satisfies the weak specification property robustly on an isolated invariant set  $\Lambda$ then $\Lambda$ is is a topologically mixing  hyperbolic set.
Thus, if $X$ is a vector field which has the weak specification property $C^1$-robustly then it generates  
a topologically mixing Anosov flow.  The authors of \cite{AST} proved first that robust specification would 
lead to sectional-hyperbolicity and then, they used robustness and perturbative techniques to rule 
out singularities and deduce uniform hyperbolicity. 

Given the current interest in a global description of dynamical systems it is natural to ask whether
the weak specification property can hold generically or at least densely in the complement of 
the set of uniformly hyperbolic flows.  Here we are able to prove that robustness assumption can
be dropped from the assumptions of \cite{AST}. We prove that every 
sectional-hyperbolic flow with specification is indeed uniformly hyperbolic (see Theorem~\ref{thm:main1}). 
This follows from an abstract criterium which asserts that any partially hyperbolic attractor for a flow 
$(X_t)_{t\in \mathbb R}$ with  specification cannot have critical elements with different indexes (see 
Theorem~\ref{thm:main2} for the precise statement).

The paper is organized as follows. In Section~\ref{Stat} we introduce some definitions and state our main
theorems and some corollaries. In Section~\ref{Aux} we prove some auxiliary lemmas that will play a key role in the proof of the main result. The proof of Theorem~B is given in Section~\ref{proofs}. Finally, in Section~\ref{proofs2} we prove the 
corollaries.

\section{Preliminaries and statement of the main results}\label{Stat}

Throughout, let $M$ be a $C^{\infty}$ compact connected boundaryless Riemannian manifold of dimension 
$\dim M\ge 3$ and let $d$ be the distance on $M$ induced from a Riemannian metric $\|\cdot\|$ on the tangent bundle $TM$. Denote by $\mathfrak{X}^1(M)$ the set of all $C^1$-vector fields on $M$
endowed with the $C^1$-topology.
Hereafter let $X\in\mathfrak{X}^1(M)$.
Then $X$ generates a $C^1$ flow $(X_t)_{t\in \mathbb R}$
on $M$. 

A set $\Lambda\subset M$ is said to be \emph{invariant} if
$X_t(\Lambda)=\Lambda$ holds for any $t\in\mathbb{R}$.
Let $\Lambda\subset$ be a compact invariant set.
We say that $\Lambda$ is \emph{transitive} if there exists a point $x\in\Lambda$
such that $\{X_t(x):t\in\mathbb{R}\}$ is dense in $\Lambda$.
$\Lambda$ is said to be an \emph{attractor} if it is transitive and there exists
an open neighborhood $U\subset M$ of $\Lambda$ such that $X_t(\bar U)\subset U$ for $t>0$ and
$\Lambda=\bigcap_{t\ge 0}X_t(U)$.

We say that $F_{\Lambda}=\{F_x\}_{x\in\Lambda}\subset T_{\Lambda}M$ is a \emph{subbundle} over $\Lambda\subset M$
if each $F_x$ is a linear subspace of $T_xM$ and a map $x\in\Lambda\mapsto F_x$ is continuous.
A subbundle $F_{\Lambda}$ over an invariant set $\Lambda$ is said to be \emph{invariant} if
$D_xX_t(F_x)=F_{X_t(x)}$ holds for every $x\in\Lambda$ and every $t\in\mathbb{R}$.

We say that a compact invariant set $\Lambda\subset M$ is a \emph{hyperbolic set} for $(X_t)_{t\in \mathbb R}$ if there exists 
a continuous invariant splitting $T_\Lambda M=F^s\oplus F^c\oplus F^u$ such that $F_x^c$ is the subspace generated by $X(x)$ and there are constants $C>0$ and $\lambda \in (0,1)$ so that  
$\|D_xX_t\mid F_x^s\| \le C \lambda^t$ and $\|(D_xX_t\mid F_x^u)^{-1}\| \le C \lambda^t$
for every $t\ge 0$ and $x\in \Lambda$. If $\Lambda$ is a hyperbolic set for the flow and $\Lambda=M$ then 
$(X_t)_{t\in \mathbb R}$ is called an \emph{Anosov flow}. 

A point $p \in M$ is a \emph{singularity} for $X$ if $X(p)=0$
and is called a \emph{regular point} otherwise.
We say that a singularity $p$ is \emph{hyperbolic} if the one-point invariant set $\{ p \}$ is a hyperbolic set.
A point $p\in M$ is \emph{periodic} if there exists a minimum period
$T>0$ so that $X_T(p)=p$ and we say that $p$ is a \emph{periodic hyperbolic point} if the orbit
$\mathcal O(p)=\cup_{t\in [0,T]} X_t(p)$ is a hyperbolic set for $X$. 
Finally, (an orbit of) a point $x$ by the flow is called a \emph{critical element} if it is either periodic or $x$ is a singularity.

We say that a compact $(X_t)_{t\in\mathbb R}$-invariant set $\Lambda\subset M$ is \emph{partially hyperbolic} if there are a continuous
invariant splitting $T_{\Lambda}M=E^s\oplus E^c$, constants $C>0$ and $\lambda\in (0,1)$ so that
\begin{equation*}
\|D_xX_t|E_x^s\|\le C\lambda^t \text{ and } \|D_xX_t|E_x^s\|~ \|D_{X_t(x)}X_{-t}|E^c_{X_t(x)}\|\le C\lambda^t
\end{equation*}
for every $x\in\Lambda$ and $t\ge 0$.
If, in addition, the following two conditions (i) and (ii) hold, then we say that $\Lambda$ is \emph{sectional-hyperbolic}:
\begin{itemize}
\item[(i)] every singularity $p\in\Lambda$ is hyperbolic;
\item[(ii)] $E^c$ is sectionally expanding, i.e. $\dim E^c\ge 2$ and $|\det(D_xX_t\mid_{L_x})|\ge C^{-1}\lambda^t$ for every $x\in\Lambda$, $t\ge 0$, and every two-dimensional subspace $L_x\subset E^c_x$.
\end{itemize}
With some abuse of notation, we say that the flow $(X_t)_{t\in \mathbb R}$ is partially hyperbolic if $M$ is a partially hyperbolic set.
Let us also mention that the notions of sectional hyperbolicity and singular-hyperbolicity coincide for three-dimensional flows, where the later arose in the characterization of robustly transitive attractors in dimension three. 
We observe that if a sectional hyperbolic flow does not have singularities 
then it is necessarily hyperbolic (see e.g. \cite{MM} for more details). 

We say that a compact $(X_t)_{t\in\mathbb R}$-invariant subset $\Lambda\subset M$ 
has the  \emph{specification property} if for any $\epsilon> 0$ there exists a $T= T(\epsilon)>0$ such that
the following property holds: given any finite collection of intervals $I_i=[a_i, b_i] \subset \mathbb R$ 
$i=1\dots m$ satisfying $a_{i+1} - b_i\ge T(\epsilon)$ for every $i$ and every map
$P:\bigcup_{i=1}^m I_i \to \Lambda$ such that $X_{t_2}(P(t_1))=X_{t_1}(P(t_2))$ for any $t_1,t_2\in I_i$ 
there exists  $x\in\Lambda$ so that $d(X_t(x), P(t)) < \epsilon$ for all $t\in \bigcup_i I_i$. 
When the previous shadowing property is required only to specifications made by two pieces of orbits 
($m=2$ above) we shall refer to this as the \emph{weak specification property}.
$\Lambda$ is said to be \emph{topologically mixing} if for all non-empty open sets $U$ and $V$ of $\Lambda$ we can take $N>0$ such that 
\begin{equation*}
U \cap X_t (V) \not= \emptyset, n\ge N.
\end{equation*}
Then it is known that topological mixing implies transitivity.
In \cite[Lemma~3.1]{AST} it was proved that if $\Lambda$ has the weak specification property then $\Lambda$ is topologically mixing.
In particular, this property implies that a flow has neither sources nor sinks in $\Lambda$.
After \cite{AST} it is natural to ask which Lorenz attractors satisfy the specification property. 
Recall that Lorenz attractors do not satisfy the shadowing property with rare exceptions (c.f. \cite{Ko}). 
Here we answer this question.

\begin{maintheorem}\label{thm:main1}
Every transitive sectional-hyperbolic attractor is either hyperbolic or does not satisfy the weak specification property.
\end{maintheorem}

If $p$ is a hyperbolic periodic point (i.e. $T_{\mathcal O(p)} M$ admits an invariant splitting $F^s\oplus F^c\oplus F^u$ as above), then \emph{the strong-stable manifold}
\begin{align*}
W^{ss}(p) 
	=\left\{x\in M : \lim_{t\to +\infty} d(X_t(x),X_t(p))=0\right\} 
\end{align*}
is indeed a $C^1$-submanifold tangent to $F^{s}$ (see \cite{HPS}).
We define the \emph{stable manifold} as 
\begin{align*}
W^{s}(p)=\bigcup_{t\in \mathbb R} X_t(W^{ss}(p)),
\end{align*}
which is a $C^1$-submanifold tangent to $F^{s}\oplus F^c$.
Let $d^{ss}$ be the distance in $W^{ss}(p)$ induced by the Riemannian metric.
The \emph{local stable manifold} at $p$ is defined by $W_\vep^s(p)=\bigcup_{|t|\le \vep} X_t(W^{ss}_\vep(p))$ where 
\begin{align*}
W_{\vep}^{ss}(p) 
	=\left\{x\in W^{ss}(p) : d^{ss}(x,p)\le \vep \right\} 
\end{align*}
for $\vep >0$. 
Moreover, observe that for $\vep>0$ there exists $\vep_0>0$ such that
\begin{align*}
\bigcap_{T\ge 0} B_{T}(p,\vep_0) \subset W_\vep^{s}(p).
\end{align*}
where $B_T(p,\vep_0)=\{ x\in M \colon d(X_t(x), X_t(p)) \le \vep_0, 0\le t \le T\}$ and consequently,  $W_\vep^{s}(p)$ contains the intersection of dynamical balls computed only for future iterates (see Lemma \ref{stable}).	
Analogously, (local) strong-unstable and unstable manifolds $W^{uu}_\vep(p)$, $W^{uu}(p)$, $W_\vep^{u}(p)$ and $W^{u}(p)$ are defined with respect to $X_{-t}$.

When $p$ is a hyperbolic singularity, we define the \emph{stable manifold} by
\begin{align*}
W^{ss}(p) 
	=\left\{x\in M : \lim_{t\to +\infty} d(X_t(x),p)=0\right\}. 
\end{align*}
Then by the stable manifold theorem we have that $W^{ss}(p)$ is a $C^1$-submanifold tangent to $F^s$. Set 
\begin{align*}
W_{\vep}^{s}(p) 
	=\left\{x\in M : d(X_t(x),p)\le \vep ~(t\ge 0) \right\} ~(\vep>0),
\end{align*}
which is called the \emph{local stable manifold}.
Then there exists $\vep_0>0$ such that
\begin{align*}
W^{ss}(p) 
	=\bigcup_{t\ge 0} X_{-t}(W^s_{\vep}(p))~~(0< \vep < \vep_0).
\end{align*}
By the definition of the singularity we have that $X_t(W^{ss}(p))=W^{ss}(p)$ for $t\in {\mathbb R}$, and so if we put $W^{s}(p)=\bigcup_{t\in \mathbb R} X_t(W^{ss}(p))$, then $W^{ss}(p)=W^{s}(p)$.
Analogously, we define (local) unstable manifolds $W_\vep^{u}(p)$ and $W^{u}(p)$ with respect to $X_{-t}$.

Observe that, by the definition of sectional hyperbolicity, all singularities are hyperbolic 
and all periodic orbits $p$ have stable index $\dim W^{ss}(p)$ equal to $\dim E^s$. 
Recently, periodic orbits for sectional-hyperbolic attractors were constructed by Lopez~\cite{Lop}, and in \cite[Proposition 10]{AM} Arbieto and Morales showed that the stable indices $\dim W^{ss}(q)$ of singularities $q$ for every nontrivial transitive sectional-hyperbolic set are equal to $\dim E^s+1$.
Moreover, every sectional-hyperbolic flow without singularities is actually hyperbolic.
Hence, Theorem~\ref{thm:main1} is actually a consequence of the more general result:

 \begin{maintheorem}\label{thm:main2}
Let $X\in \mathfrak{X}^1(M)$ be a vector field and let $\Lambda$ be an attractor 
so that the flow $(X_t)_{t\in\mathbb R}$ admits a partially hyperbolic splitting $T_\Lambda M=E^s\oplus E^c$. Assume that there are two hyperbolic critical elements $p$ and $q$ such that 
${\rm dim}\;E^s={\rm dim}\;W^{ss}(p)<{\rm dim}\;W^{ss}(q)$.
Then $X\mid_\Lambda$ does not satisfy the weak specification property.
\end{maintheorem}


Now we briefly describe the geometric Lorenz attractor.
Let $\Sigma=\{(x,y,1)\in\mathbb{R}^3:|x|,|y|\le 1\}$ and $\Gamma=\{(0,y,1)\in\mathbb{R}^3:|y|\le 1\}$.
A $C^1$-vector field $X$ on ${\mathbb R}^3$ is said to be a {\it geometric Lorenz vector field} if it satisfies the following conditions
\begin{enumerate}
\item[(1)] For any point $(x,y,z)$ in a neighborhood of the origin ${\bf 0}$ of ${\mathbb R}^3$, $X$ is given by $(\dot{x}, \dot{y}, \dot{z})=(\lambda_1 x, -\lambda_2 y, -\lambda_3 z)$ where $0<\lambda_3<\lambda_1<\lambda_2$.
\item[(2)] All forward orbits of $X$ starting from $\Sigma\setminus \Gamma$ will return to $\Sigma$ and the first return map $L:\Sigma \setminus \Gamma\to \Sigma$ 
is a piecewise $C^1$ diffeomorphism which has the form 
\begin{equation*}
L(x,y,1)=(\alpha (x), \beta (x,y),1),
\end{equation*}
where $\alpha:[-1,1]\setminus \{ 0\} \to [-1,1]$ is a piecewise $C^1$-map with $\alpha(-x)=-\alpha(x)$ and satisfying
\begin{equation*}
\left\{ 
\begin{array}{ll}
\lim_{x\to 0+} \alpha(x)=-1, &\alpha (1)<1,\\
\lim_{x\to 0+} \alpha'(x)=\infty, &\alpha' (x)>\sqrt{2} \text{ for any } x\in (0,1].
\end{array}
\right.
\end{equation*}
\end{enumerate}
A $C^1$-map $X_t:{\mathbb R}^3 \times {\mathbb R} \to {\mathbb R}^3$ is the geometric Lorenz flow if it is generated by a geometric Lorenz vector field $X$ (see e.g. \cite{Gu, GuWi, KS} for more details).
Let $T_X$ be the closure of the set $\displaystyle\bigcup_{t\ge 0}X_t(\Sigma\setminus\Gamma)$ in
$\mathbb{R}^3$ and
set $\displaystyle\Lambda=\bigcap_{t\ge 0}X_t(T_X)$. Then it is known that
$\Lambda$ is a partially hyperbolic attractor (see \cite{AP} for details).
We call $\Lambda$ the \textit{geometric Lorenz attractor}.

The stable index of the singularity $q$ of the geometric Lorenz attractors satisfies $\dim W^{ss}(q)= \dim E^s+1$. 
Hence, we obtain the following immediate consequence:

\begin{maincorollary}
\label{Lorenz}
Assume that $(X_t)_{t\in \mathbb R}$ is a flow on ${\mathbb R}^3$ that admits a geometric Lorenz attractor $\Lambda$. 
Then $(X_t)_{t\in \mathbb R}$ does not satisfy the weak specification property on $\Lambda$.
\end{maincorollary}

\begin{remark}
Even though Theorem \ref{thm:main2} is proved for compact manifolds, it applies to geometric Lorenz attractors because they can be viewed as the restriction of flows on a 3-sphere. 
For that reason, compact Riemannian manifolds of dimension larger or equal to 3 admit vector fields that exhibit 
geometric Lorenz attractors (see e.g. Subsection 3.3 in \cite{AP}).
\end{remark}

We notice that if $\dim M=3$ then every $C^1$-robustly transitive set with singularities 
$\Lambda$ is a singular-hyperbolic set up to flow-reversing~\cite{MPP} and consequently, the flow $(X_t)_{t\in \mathbb R}$ does not satisfy the specification property on $\Lambda$.
Observe that the previous theorem also applies to partially hyperbolic sets $\Lambda$ with a 
decomposition $E^u\oplus E^{c}$ just by considering the vector field $-X$.  
Moreover, even in the case of an Anosov flow $(X_t)_{t\in \mathbb R}$ the time-1 map $f=X_1: M \to M$ 
of an Anosov flow is a strongly partially hyperbolic diffeomorphism that admits no hyperbolic periodic points. In
particular an analogous theorem as the previous one for flows does not follow from the ones obtained
for partially hyperbolic diffeomorphisms in \cite{SVY}. 
Nevertheless some corollaries of the main result in \cite{SVY} for strongly partially hyperbolic diffeomorphisms
on three-manifolds can be expected to hold for strongly partially hyperbolic flows
on four-manifolds due to the neutral direction of the vector field. We shall discuss now such extensions.

We say that a flow is \emph{strongly partially hyperbolic} with $d$-dimensional central direction ($d\ge 1$) if 
there are a continuous invariant splitting $TM=E^s\oplus E^c \oplus E^u$ with $\dim E^c =d$, constants $C>0$ and $\lambda\in (0,1)$
so that
\begin{eqnarray*}
&&\|D_xX_t|E_x^s\| \le C\lambda^t,~~\|(D_xX_t|E_x^u)^{-1}\|\le C\lambda^t,\\
&&\|D_xX_t|E_x^s\|~\|D_{X_t(x)}X_{-t}|E^c_{X_t(x)}\|\le C\lambda^t \text{ and } \\
&&\|D_xX_t|E_x^c\|~\|D_{X_t(x)}X_{-t}|E^u_{X_t(x)}\|\le C\lambda^t 
\end{eqnarray*}
for every $x\in M$ and $t\ge 0$.
Denote by $\mathcal{SPHF}_d(M)$ the set of such flows and note that it is an open subset of 
$\mathfrak{X}^1(M)$.
We say a flow $(X_t)_{t\in \mathbb R}$ generated by a vector field $X$ is \textit{robustly transitive} if all flows generated by vector fields in a $C^1$-open neighborhood of $X$ are transitive, that is, have a dense orbit. If the vector field $X$ has an attractor
$\Lambda_X:=\bigcap_{t\ge 0} X_t(U)$ we say that $\Lambda$ is a \emph{robustly transitive attractor} if  
for any vector field $Y$ in a $C^1$-open neighborhood of $X$ the attractor $\Lambda_Y:=\bigcap_{t\ge 0} Y_t(U)$
is transitive.
Finally, we denote by $\mathcal{RNTF}$ the set of robustly non-hyperbolic transitive flows
(that is, flows generated by vector fields $X$ so that every $C^1$-vector field $Y$ in a
$C^1$-neighborhood of $X$ generates a non-hyperbolic and transitive flow) endowed with 
the $C^1$-topology in the space of vector fields.


 %
In the case that the central direction $E^c$ is two dimensional, any two hyperbolic periodic points 
with different indices verify the assumptions of Theorem~\ref{thm:main2}. Thus
we obtain the following direct consequence.

\begin{maincorollary}\label{cor1}
Let $X\in \mathcal{SPHF}_2(M)$ and suppose that there exist two  hyperbolic critical elements
with different indices. Then $X$ does not satisfy the weak specification property.
\end{maincorollary}

Since $X(x)$ is in the central direction $E^c$ for a nonsingular partially hyperbolic flow $(X_t)_{t\in \mathbb R}$, we can obtain the following corollary in a similar way as above.

\begin{maincorollary}\label{c3}
Let $X\in \mathcal{SPHF}_3(M)$. If $X$ is nonsingular and if there exist two  hyperbolic critical elements with different indices, then $X$ does not satisfy the weak specification property.
\end{maincorollary}

Using $C^1$-perturbative techniques one can show that hyperbolic flows coincide with 
the class star-flows $\mathcal {G}^1(M)$ 
(i.e. flows such that all critical elements are hyperbolic   
$C^1$-robustly) (see e.g. \cite{AST} for a more precise description). We deduce that, from the topological 
viewpoint, most robustly non-hyperbolic and transitive partially hyperbolic flows with three dimensional central direction do not have the specification property. More precisely,

\begin{maincorollary}\label{nh}
There is a $C^1$-open and dense subset $O$ in $\mathcal{RNTF}\cap\mathcal{SPHF}_3(M)$, such 
that every $X\in O$ does not satisfy the weak specification property.
\end{maincorollary}

We can expect to extend the previous result by removing the partial hyperbolicity assumption
in a lower dimensional setting. In the case that $\dim M=3$, Doering~\cite{Doering} proved that 
every $C^1$-robustly transitive flow on a three-dimensional manifolds is Anosov and consequently 
satisfies the specification property.  If $\dim M=4$ we can remove the assumption of partial hyperbolicity
from the previous corollary.

\begin{maincorollary}
\label{dim4}
Suppose that $\dim M=4$. Then there is a $C^1$-open and dense subset $O$ in $\mathcal{RNTF}$ 
so that every $X\in O$ does not satisfy the weak specification property.
\end{maincorollary}

 Let us remark that Komuro~\cite{Ko} proved that the Lorenz attractors do not satisfy the shadowing property.
It follows from our results that these attractors do not satisfy the specification property neither. 
 Several authors considered recently either measure theoretical non-uniform specification properties (see e.g. \cite{OT,Va12}) or almost specification properties (see e.g. \cite{PS05,Th10}) to the study of the ergodic properties of a dynamical system. One remaining interesting question is to understand which partially hyperbolic flows admit weaker specification properties. A global picture
that includes the characterization of dynamical systems satisfying these weaker kinds of specification is still incomplete.

\section{Auxiliary results}\label{Aux}

In this section we provide necessary definitions and prove some auxiliary results used in the proofs of the main results.
The first is a well known result whose proof we shall include for the reader's convenience.

\begin{lemma}
\label{unstable}
Let $\Lambda$ be an attractor. 
Then, for every hyperbolic critical element $p\in \Lambda$, we have
\begin{equation*}\label{}
W^{uu}(p)\subset \Lambda.
\end{equation*}
In particular, we have that $W^{u}_{\vep}(p)\subset \Lambda$ for every $\vep>0$.
\end{lemma}

\begin{proof}
Since $\Lambda$ is an attractor let $U \subset M$ be an open neighborhood so that $X_t(\bar U) \subset U$ and
$\Lambda=\bigcap_{t\ge 0}X_t(U)$. If $p \in \Lambda$ is a hyperbolic periodic orbit for $(X_t)_t$ there are constants
$C_p>0$ and $\lambda_p \in (0,1)$ so that 
$$
d(X_{-t}(x),X_{-t}(p)) \le C_p \lambda_p^t d(x,p)
$$ 
for every $x\in W^{uu}_\vep(p)$ and $t\ge 0$.
Using this backward contraction and that $U$ is an open set, there exists a small $\vep>0$ so that 
$X_{-t}(W^{uu}_\vep(p)) 
\subset U$ for every $t\ge 0$, which proves that $W^{uu}_\vep(p) \subset \Lambda$. The $(X_t)_t$-invariance of $\Lambda$
and the equalities $W^{uu}(p)=\bigcup_{t\ge 0} X_t( W_\vep^{uu}(X_{-t}(p)) )$ and $W^{u}(p)=\bigcup_{t\ge 0} X_t (W^{uu}(p))$ guarantee that both $W^{uu}(p)$ and $W^{u}(p)$ are contained in $\Lambda$. Since the proof in the case that $p$ is a singularity 
is completely analogous we shall omit it. 
\end{proof}

\color{black}
\begin{lemma}
\label{stable}
For every hyperbolic periodic point $p$ and $\vep>0$, we can choose $\epsilon_0\in (0,\epsilon)$ such that for $x\in M$, if $d(X_t(x),X_t(p))\le \epsilon_0$ for every $t\ge 0$ then
\begin{equation*}\label{}
x\in W^s_{\vep}(p)=\bigcup_{|t|\le \vep} X_t(W^{ss}_{\vep}(p)).
\end{equation*}
\end{lemma}

\begin{proof}
Let $\pi(p)>0$ be the prime period of the periodic point $p$.
We set $\Gamma:=\bigcup_{t\in [0,\pi(p)]} X_t(p)$. 
Since $p$ is hyperbolic, there exist a continuous invariant splitting $T_{\Gamma}M=F^s\oplus F^c \oplus F^u$, 
constants $\lambda_1\in
(0,1)$ and $C>0$ such that $F^c_x$ is generated by $X(x)$ and
\begin{equation}
\label{hypcritical}
\|D_xX_t|F_x^s\|\le C\lambda_1^t,\ \|(D_xX_t|F_x^u)^{-1}\|\le C\lambda_1^t
\end{equation}
for any $t\ge 0$ and $x\in\Gamma$.
It follows from \cite[Lemma 4.4]{HPPS} that there exist a neighborhood $U'$ of $\Gamma$
and a continuous splitting $T_{U'}M=\tilde{F^s}\oplus \tilde{F^c}\oplus\tilde{F^u}$
such that $\tilde{F_x^{\sigma}}=F_x^{\sigma}$ ($\sigma=s,c,u$) whenever $x\in\Gamma$.

For $x\in U'$, $\kappa>0$, we define the unstable cone field
$$
C_{\kappa}^u(x)=\{ v=v_1+v_2\in (\tilde{F}_x^s\oplus \tilde{F}_x^c) \oplus \tilde{F}_x^u : \| v_1\| \le \kappa \| v_2\| \}.
$$
By the equation (\ref{hypcritical}), there are $\kappa>0$, $0<\lambda_2<1$ and $T>0$ 
with $X_T(p)=p$ such that if $x\in\Gamma$, then 
$$D_xX_T(C_{\kappa}^u(x)) \subset C_{\frac{\kappa}{2}}^u(X_T(x)),\ 
\|D_xX_T (v)\| \ge \lambda_2^{-1}\| v\|\ (v\in C^u_{\kappa}(x)).$$
Since the splitting $T_{U'}M=\tilde{F^s}\oplus \tilde{F^c}\oplus\tilde{F^u}$ is continuous,
we can find a neighborhood $U\subset U'$ of $\Gamma$
\color{black}
and $0<\lambda<1$ such that if $X_s(x)\in U$ for $0\le s\le T$, then 
\begin{eqnarray}
D_xX_T(C_{\kappa}^u(x)) &\subset& C_{\kappa}^u(X_T(x)),~~\nonumber\\
\|D_xX_T (v)\| &\ge& \lambda^{-1}\| v\|~~(v\in C^u_{\kappa}(x)).\label{Cu}
\end{eqnarray}
Increasing $T$ if necessary we may assume that $X_{T} (W^{s}_{\vep}(p))\subset  W^{s}_{\vep}(p)$.
Choose $\delta_0>0$ (depending on $T$) such that if $d(x,p)\le \delta_0$, then $X_{t}(x)\in U$ for $0\le t\le T$.
Since $T_{p}W^{s}(p)=F^s_p\oplus F^c_p$, we have that $W^s_{\vep}(p)$ is a $C^1$ disk with $T_pW^{s}_{\vep}(p)=F^s_p\oplus F^c_p$.
So we can take $0<\vep_0<\theta<\delta_0/2K$ (where we set
$K:=\{\|D_xX_T\|:x\in M\}<\infty$) such that if $d(x,p)\le \vep_0$, then the following hold:
\begin{enumerate}

\item[$(1)$]
There is a $C^1$ disk $D\subset U$ centered at $x$ of radius $\theta$ such that 
\begin{eqnarray}\label{DCu}
&&\dim D =\dim F^u_p \text{ and } T_yD \subset C_{\kappa}^u(y)~~(\text{for all } y\in D).
\end{eqnarray}

\item[$(2)$]
Any disk centered at $x$ of radius $r$ with $\theta\le r\le K\theta$ satisfying (\ref{DCu}) intersects $W^s_{\vep}(p)$ at a unique point transversely. Such an intersection point $y$ satisfies 
\begin{eqnarray}\label{Dy}
d(y,p)\le d(y,x)+d(x,p)\le K\theta+\vep_0< \delta_0.
\end{eqnarray}

\end{enumerate}

\color{black}
Assume that $x\in M$ satisfies $d(X_t(x),X_t(p))\le \epsilon_0$ for $t\ge 0$. Let $D_0$ be a $C^1$ disk centered at $x$ of radius $\theta$ satisfying (\ref{DCu}) and $y$ be the intersection of $D_0$ and $W^s_{\vep}(p)$ (see (\ref{Dy})). Since $D_0$ is contained in a ball centered at $p$ with radius $\delta_0$, we have $X_T(D_0)\subset U$ for $0\le t\le T$. By (\ref{Cu}) and (\ref{DCu}), $X_T(D_0)$ contains a $C^1$ disk centered at {$X_T(x)$} of radius $\lambda^{-1} \theta$ satisfying (\ref{DCu}). 
Denote by $D_1$ a $C^1$ disk centered at {$X_T(x)$} of radius $\theta$ contained in {$X_T(D_0)$}. Since {$X_T(y)\in X_{T} (W^{s}_{\vep}(p))\subset W^{s}_{\vep}(p)$} and since both $D_1$ and {$X_T(D_0)$} intersect $W^{s}_{\vep}(p)$ at a unique point respectively, we have 
$$
\{ X_T(y) \} =X_T(D_0)\cap W^{s}_{\vep}(p) =D_1\cap W^{s}_{\vep}(p).
$$
Moreover, since {$X_{T}(x), X_T(y)\in D_1$}, we have
\begin{eqnarray*}
d(x,y)&=&d(X_{-T}(X_{T}(x)),X_{-T}(X_{T}(y)))\nonumber\\
&\le & \lambda d(X_{T}(x),X_{T}(y)) \le \lambda \theta. 
\end{eqnarray*}

Repeating this procedure, we find $C^1$ disks $D_n$ $(n\ge 0)$ centered at $X_{nT}(x)$ of radius $\theta$ satisfying (\ref{DCu}) such that 
$$D_{n+1}\subset X_T(D_n) \text{ and } X_{nT}(x), X_{nT}(y)\in D_n$$
for $n\ge 0$. So, for every $n\ge 0$
\begin{eqnarray*}
d(x,y)&=&d(X_{-nT}(X_{nT}(x)),X_{-nT}(X_{nT}(y)))\nonumber\\
&\le& \lambda^n \theta,
\end{eqnarray*}
which means $d(x,y)=0$. So $x\in W^s_{\vep}(p)$, which finishes the proof.
\end{proof}

\begin{rem}
\label{unstable}
An analogous result holds for the local unstable manifold as follows: for every hyperbolic periodic point $p$ and $\vep>0$, we can choose $\epsilon_0\in (0,\epsilon)$ such that for $x\in M$, if $d(X_t(x),X_t(p))\le \epsilon_0$ for $t\le 0$, then $x\in W^u_{\vep}(p)$.
\end{rem}

\begin{lemma}
\label{dense}
Let $\Lambda$ be an attractor and suppose that $\Lambda$ has the weak 
specification property. 
Then for every hyperbolic  critical element $p\in \Lambda$, the strong stable manifold $W^{ss}(p)$ is dense in $\Lambda$.
\end{lemma}

\begin{proof}
We consider only the case when $p$ is periodic since the singularity case can be shown similarly.
Let $\vep>0$ and $z\in \Lambda$ be fixed arbitrarily.
Since $(X_t)_{t\in \mathbb R}$ is the flow generated by the vector field $X$ we can take $0<t_0<\epsilon$ 
 so that $d(x,X_t(x))\le \epsilon$ for any $x\in\Lambda$ and $|t|\le t_0$.
By Lemma \ref{stable} we can choose $\epsilon_0\in (0,t_0)$ such that
if $d(X_t(x),X_t(p))\le \epsilon_0$ for every $t>0$ then
\begin{equation}\label{Wst}
x\in W^s_{t_0}(p)=\bigcup_{|t|\le t_0} X_t(W^{ss}_{t_0}(p)).
\end{equation}

Let $T(\epsilon_0)>0$ be as in the definition of the specification property
and choose $T\ge T(\epsilon_0)$ so that {$X_T(p)=p$}.
By the weak specification property, there are $x_n\in\Lambda$ so that $d(x_n,z) \le \epsilon_0$ 
and {$d(X_t(X_T(x_n)),X_t(p))\le\epsilon_0$} for every $t\in [0,n]$.
By compactness of $\Lambda$, we may assume that $(x_n)_{n\in \mathbb N}$ is convergent to some point $x\in\Lambda$ satisfying $d(x,z)\le\epsilon_0$ and {$d(X_t(X_T(x)),X_t(p))\le\epsilon_0$} for every $t>0$.
Using (\ref{Wst}), we have
$$X_T(x)\in\bigcup_{|t|\le t_0} X_t(W^{ss}_{t_0}(p))$$
and we can find $t_1\in [-t_0,t_0]$ such that 
{$X_T(x)\in X_{t_1}(W^{ss}_{t_0}(p))$}.
Since $T$ is the period of $p$, we have $x\in X_{t_1}(W^{ss}(p))$.
Thus, there exists a point $y\in W^{ss}(p)$ such that $x=X_{t_1}(y)$ 
and consequently
$$
d(y,z)\le d(y,x)+d(x,z)\le \epsilon+\epsilon_0\le 2\epsilon,
$$
which implies that $W^{ss}(p)$ is dense in $\Lambda$.
\end{proof}

\begin{lemma}\label{TWu}
Let $\Lambda$ be a partially hyperbolic attractor with splitting $T_{\Lambda}M=E^s \oplus E^c$ and let $q\in \Lambda$ be a hyperbolic critical element. Then we have $T_xW^u(q) \subset E^c_x$ for every $x\in W^u(q)$.
\end{lemma}

\begin{proof}
We deal with the case when $q$ is a hyperbolic periodic point.
Let $\pi(q)>0$ be the prime period of $q$.
To reach a contradiction we assume that there exist $x\in W^u(q)$ and $v \in T_xW^u(q) \setminus E^c_x$. 
Since $x\in W^u(q)=\bigcup_{t\in \mathbb R} X_t(W^{uu}(q))$, we can choose $y\in{\mathcal O(q)}$ such that $x\in W^{uu}(y)$. 
If we put $t_n=n \pi(q)$ for $n\in {\mathbb N}$, then $d(X_{-t_n}(x), y)\to 0$ as $n\to \infty$. 
By the $(X_{t})_t$-invariance of $W^{u}(q)$, we have 
\begin{equation}\label{Wuq}
D_xX_{-t_n}(T_xW^{u}(q)) \to T_yW^{u}(q)~~(n\to \infty).
\end{equation}

Since $v \in T_xW^u(q) \setminus E^c_x$, we can take $v_s\in E^s_x \setminus \{ 0\}$ and $v_c\in E^c_x$ such that $v=v_s+v_c$. 
By the definition of partial hyperbolicity, we have $D_xX_{-t_n}(v_s)\in E^s_{X_{-t_n}(x)}$, $D_xX_{-t_n}(v_c)\in E^c_{X_{-t_n}(x)}$ and
\begin{eqnarray*}
\| D_xX_{-t_n} (v_c) \| / \| D_xX_{-t_n} (v_s) \| 
&\le&  \| D_xX_{-t_n}|E^c\| \|v_c\| / \| D_{X_{-t_n}(x)} X_{t_n}|E^s\|^{-1} \| v_s \| \nonumber\\
&\le&  \| D_{X_{-t_n}(x)} X_{t_n}|E^s\| \| D_xX_{-t_n}|E^c\| (\|v_c\| /  \|v_s \|)\nonumber\\
&\le&  C^{1} \lambda^{t_n} (\|v_c\| /  \|v_s \|)\to 0 ~~(n\to \infty).
\end{eqnarray*}
Without loss of generality we may assume that $D_{x}X_{-t_n}(v)/\| D_{x}X_{-t_n}(v)\|$ converges to some unit vector $v'\in E_y^s$.
By (\ref{Wuq}) we have $v'\in T_yW^{u}(q) \cap E_y^s$.

By the hyperbolicity of $q$ and $y\in {\mathcal O}(q)$, there exists $0<\lambda_0 <1$ such that $\| D_y X_{-t} | T_yW^{uu}(y) \| \le C\lambda_0^t$ for $t\ge 0$.
Since $W^{u}(q)=\bigcup_{t\in \mathbb R} X_t(W^{uu}(y))$, we have 
$T_yW^u(q)=\langle X(y) \rangle \oplus T_yW^{uu}(y)$. Here $\langle X(y) \rangle$ denotes the one dimensional subspace generated by $X(y)$.
So we can find $v_1\in \langle X(y) \rangle$ and $v_2\in T_yW^{uu}(y)$ such that $v'=v_1+ v_2$. Then 
\begin{eqnarray*}
\| D_yX_{-t_n} (v') \| &\le&  \| D_yX_{-t_n} (v_1) \| + \| D_yX_{-t_n} (v_2) \| \\
&\le&  \| v_1 \| + \| D_yX_{-t_n}|T_yW^{uu}(y)\| \|v_2\| \\
&\le&  \| v_1 \| + C\lambda_0^{t_n} \|v_2\| \\
&\to&  \| v_1 \| ~~(n\to \infty).
\end{eqnarray*}
However, since $v' \in E^s_y$, we have
\begin{eqnarray*}
\| D_yX_{-t_n} (v') \| &\ge& \| D_{y} X_{t_n}|E^s\|^{-1} \| v' \| \ge  C^{-1} \lambda^{-t_n} \| v' \| \to \infty ~~(n\to \infty),
\end{eqnarray*}
which is a contradiction.
\end{proof}

\begin{lemma}\label{cdisk}
Let $\Lambda$ be as above and let $q\in \Lambda$ be a hyperbolic critical element. Then for every $x\in W^u(q)$ there exist a neighborhood $U(x)$ of $x$ and a $C^1$-disk $D(x) \subset M$ such that 
\begin{eqnarray*}
x\in W^u_x(q) \subset D(x) \text{ and } T_xD(x)=E^c_x.
\end{eqnarray*}
where $W^u_x(q)$ denotes the connected component of $W^u(q) \cap U(x)$ containing $x$.
\end{lemma}

\begin{proof}
Let $x\in W^u(q)$ and put $F_x=T_xW^u(q)$.
Then by Lemma \ref{TWu} we have $F_x \subset E^c_x$.
So we can take a subspace $G_x\subset E^c_x$ such that $F_x \oplus G_x =E^c_x$.
By the definition of $F_x$ there exist $\vep>0$ and a $C^1$-map $\psi:F_x \cap \{ \| v\|\le \vep \} \to G_x\oplus E^s_x$ with $D_x\psi=0$ such that 
\begin{eqnarray}\label{Vx}
V(x)=\exp_x \{ v+\psi(v) : v\in F_x,~\| v \| \le \vep  \} \subset W^u(p).
\end{eqnarray}
If we put 
\begin{eqnarray*}
D(x)=\exp_x \{ v+v'+ \psi(v) : v\in F_x,~ \| v \| \le \vep,~ v'\in G_x,~ \| v' \| \le \vep \},
\end{eqnarray*}
then $D(x)$ is a $C^1$-disk with $T_xD(x)=F_x \oplus G_x(=E^c_x)$  for sufficiently small $\vep>0$.
Let $U(x)=\exp_x \{ v : \| v\| < \vep/2\}$.
Then, by (\ref{Vx}), we have 
\begin{eqnarray*}
x\in W^u_x(q)=V(x) \cap U(x) \subset D(x).
\end{eqnarray*}
\end{proof}

We finish this section with some considerations on the existence of (local) stable foliations and holonomies.
In the rest of this section, let $\Lambda$ be a (transitive) partially hyperbolic attractor with splitting $T_\Lambda M= E^s\oplus E^c$. Without loss of generality, we may assume that the metric $\| \cdot \|$ is adapted, i.e. for all $x\in \Lambda$ 
\begin{equation*}
\| D_x X_t | E^s_x \|< 1 \text{ and } \| D_x X_t | E^s_x \|~ \| D_{X_t(x)} X_{-t} |E^c_{X_t(x)} \| <1~~(t>0)
\end{equation*}
(see \cite[Theorem 4]{G} for the existence of adapted metrics). 
Since $\Lambda$ is compact, we can take $0<\lambda'<1$ such that, for every $x\in \Lambda$,
\begin{equation}\label{adapted1}
\| D_x X_1 | E^s_x \|< \lambda' \text{ and } \| D_x X_1 | E^s_x \|~ \| D_{X_1(x)} X_{-1} |E^c_{X_1(x)} \| <\lambda'.
\end{equation}

\begin{lemma} \cite[Proposition 2.3]{Mane0} \label{s-disk}
There exists a continuous family of $C^1$-disks $\{ \mathcal{F}^s_{loc}(x) \}_{x\in \Lambda}$ such that
\begin{enumerate}
\item[(1)] $T_x\mathcal{F}^s_{loc}(x)=E^s_x$ for every $x\in \Lambda$.
\item[(2)] 
$X_1(\mathcal{F}^s_{loc}(x)) \subset \mathcal{F}^s_{loc}(X_1(x))$ for every $x\in \Lambda$.
\item[(3)] 
$\| D_{y}X_1|T_y\mathcal{F}^s_{loc}(x) \|< \lambda'$ for every $y\in \mathcal{F}^s_{loc}(x)$ and $x\in \Lambda.$ 
In particular, for $y\in \mathcal{F}^s_{loc}(x)$, we have $d(X_{n}(x), X_{n}(y))\to 0$ as $n\to \infty$.
\end{enumerate}
\end{lemma}
We set
\begin{equation*}
\mathcal{F}^s(x)=\bigcup_{n=1}^{\infty} X_{-n}(\mathcal{F}^s_{loc}(X_n(x)))~~(x\in \Lambda).
\end{equation*}
Then, by Lemma \ref{s-disk}, we can check that for $x\in \Lambda$
\begin{enumerate}
\item[$\bullet$] $X_1(\mathcal{F}^s(x))=\mathcal{F}^s(X_1(x))$,
\item[$\bullet$] $T_x\mathcal{F}^s(x)=E^s_x$ and 
\item[$\bullet$] if $y\in \mathcal{F}^s(x)$, then $d(X_t(x),X_t(y))\to 0$ as $t\to \infty$.
\end{enumerate}
Moreover, 
the following holds.
\begin{lemma} \label{invariance}
Let $\{ \mathcal{F}^s(x) \}_{x\in \Lambda}$ be as above. Then the following hold:
\begin{enumerate}
\item[$(1)$] For $x\in \Lambda$, $t\in {\mathbb R}$, we have 
\begin{equation*}
X_t(\mathcal{F}^s(x)) = \mathcal{F}^s(X_t(x)).
\end{equation*}
\item[$(2)$] If $\mathcal{F}^s(x)\cap \mathcal{F}^s(y)\not=\emptyset$ for $x, y\in \Lambda$, then we have $\mathcal{F}^s(x)=\mathcal{F}^s(y)$.
\end{enumerate}
\end{lemma}
\begin{proof}
It follows from \cite[Lemma 4.4]{HPPS} that there exist a neighborhood $U'$ of $\Lambda$
and a continuous splitting $T_{U'}M=\tilde{E^s}\oplus \tilde{E^c}$
such that $\tilde{E_x^{\sigma}}=E_x^{\sigma}$ ($\sigma=s,c$) whenever $x\in\Lambda$.
For $x\in U'$, $\kappa>0$, we define
\begin{eqnarray}
C_{\kappa}^s(x)&=&\{ v=v_1+v_2\in \tilde{E}_x^s\oplus \tilde{E}_x^c : \| v_2\| \le \kappa \| v_1\| \} \text{ and } \nonumber\\
C_{\kappa}^c(x)&=&\{ v=v_1+v_2\in \tilde{E}_x^s\oplus \tilde{E}_x^c : \| v_1\| \le \kappa \| v_2\| \}. \nonumber
\end{eqnarray}
Let $\vep>0$ be such that $e^{-5\vep}>\lambda'$ and assume that $\kappa>0$ is small so that the inequalities hold:
\begin{eqnarray}
\|D_xX_1 (v)\| &\le& e^{\vep} \|D_{x}X_{1}|\tilde{E}^s_x\|~ \| v\|~ (v\in C^s_{\kappa}(x), x\in U')  \text{ and } \label{Cs}\\
\|D_xX_1 (v)\| &\ge& e^{-\vep} \|D_{X_1(x)}X_{-1}|\tilde{E}^c_x\|^{-1}~ \| v\|~ ( v\in C^c_{\kappa}(x), x\in U').\label{Cc}
\end{eqnarray}

Since $\tilde{E_x^c}$ and $\tilde{E_x^s}$ are $DX_1$-invariant and satisfy (\ref{adapted1}) for points $x\in \Lambda$, we can choose a small neighborhood $U\subset U'$ of $\Lambda$ satisfying that for $x\in U$
\begin{equation}\label{CcCc}
D_xX_1(C_{\kappa}^c(x)) \subset C_{\lambda' \kappa}^c(X_1(x)).
\end{equation}
It follows from (\ref{Cs}) and (\ref{Cc}) that there is $\delta'>0$ such that if $x\in \Lambda$, $y\in M$ and $d(x,y)\le \delta'$, then $y\in U$ and 
\begin{eqnarray}
\|D_yX_1 (v)\| &\le& e^{2\vep} \|D_{x}X_{1}|\tilde{E}^s_x\| ~ \| v\|~ (v\in C^s_{\kappa}(y)), \label{Cs2}\\
\|D_yX_1 (v)\| &\ge& e^{-2\vep} \|D_{X_1(x)}X_{-1}|\tilde{E}^c_x\|^{-1} ~ \| v\|~ (v\in C^c_{\kappa}(y)).\label{Cc2}
\end{eqnarray}

To prove (1), we put {$\mathcal{F}_\tau(x)=X_{-\tau}(\mathcal{F}_{loc}^s(X_{\tau}(x)))$} for $x\in \Lambda$ and $\tau\in {\mathbb R}$.
By the $DX_{\tau}$-invariance of ${E}^s_x$, we have that 
\begin{equation*}
T_{x}\mathcal{F}_{\tau}(x)=\tilde{E}^s_{x}~~(\tau\in {\mathbb R}, x\in \Lambda).\end{equation*}
Thus there exists $0<\delta<\delta'$ such that if $0\le \tau \le 1$, $x\in \Lambda$ and 
{$y\in \mathcal{F}_{\tau}(x)$}
with $d(x,y)\le \delta$, then 
\begin{equation}\label{TFsC}
T_{y}\mathcal{F}_{\tau}(x)\in C_{\kappa}^s(y).
\end{equation}
By Lemma \ref{s-disk} we remark that
\begin{equation}\label{XnFt}
X_n(\mathcal{F}_\tau(x))\subset \mathcal{F}_\tau(X_n(x))
\end{equation}
for $x\in \Lambda$ and $n\in \mathbb N$.
We can show that for {$\tau, \rho\in [0,1]$} and $x\in \Lambda$
\begin{equation}\label{FstB}
\mathcal{F}_{\tau}(x) \cap B_{\delta/4}(x)=\mathcal{F}_{\rho}(x) \cap B_{\delta/4}(x). 
\end{equation}
Indeed, to reach a contradiction we assume that the equality (\ref{FstB}) does not hold. By (\ref{TFsC}) there exist 
$y\in \mathcal{F}_{\tau}(x) \cap B_{\delta/4}(x)$, 
$z\in \mathcal{F}_{\rho}(x) \cap B_{\delta/4}(x)$
 $(y\not= z)$ and a $C^1$-disk $D\subset B_{\delta/2}(x)$ containing $y$ and $z$ such that
\begin{equation*}\label{}
T_{w}D\in C_{\kappa}^c(w) ~(w\in D).
\end{equation*}
Then by (\ref{adapted1}), (\ref{Cs2}) and (\ref{TFsC})
\begin{eqnarray}\label{dnyz}
&&\max\{ d(X_n(y), X_n(x)), d(X_n(z), X_n(x)) \}\le (e^{2\vep} \lambda')^n \delta/4
\end{eqnarray}
for $n=1,2,\cdots$.
Let $D_n$ be the connected component of $X_n(D)\cap B_{\delta/2}(X_n(x))$ containing $X_n(y)$ $(n=1,2,\cdots)$. Then by (\ref{dnyz}) and (\ref{CcCc}) we have that $D_n$ contains $X_n(z)$ and $T_{w}D_n\in C_{\kappa}^c(w)$ for $w\in D_n$.
Thus by (\ref{Cc2})
\begin{equation*}\label{lower}
d(X_n(y), X_n(z))\ge \left(\prod_{i=1}^{n} e^{-2\vep} \|D_{X_i(x)}X_{-1}|\tilde{E}^c\|^{-1} \right) d(y,z).
\end{equation*}

On the other hand, by {(\ref{Cs2}), (\ref{TFsC}) and (\ref{XnFt})}, we have 
\begin{eqnarray*}
d(X_n(y), X_n(z)) 
&\le &d(X_n(y), X_n(x)) + d(X_n(x), X_n(z)) \\
&\le & \left(\prod_{i=0}^{n-1} e^{2\vep} \|D_{X_i(x)}X_{1}|\tilde{E}^s_x\|\right) \delta.
\end{eqnarray*}
Therefore 
\begin{eqnarray*}
0<d(y,z)&\le& \left( \prod_{i=1}^{n} e^{4\vep} \|D_{X_i(x)}X_{1}|\tilde{E}^s_x\| \|D_{X_i(x)}X_{-1}|\tilde{E}^c\| \right) \delta\\
&\le &(e^{4\vep} \lambda')^n \delta\to 0~(n\to \infty),
\end{eqnarray*}
which is a contradiction. Thus \eqref{FstB} holds.

By Lemma \ref{s-disk}, there exists $m\in {\mathbb N}$ such that 
$$
X_{m}(\mathcal{F}_{\tau}(y)) \subset \mathcal{F}_{\tau}(X_m(y)) \cap B_{\delta/4}(X_m(y))
$$ 
for $y\in \Lambda$ and $\tau \in [0,1]$.
By (\ref{FstB}) we can check that for any $x\in \Lambda$ and $\tau\in [0,1]$
\begin{eqnarray*}
X_{-\tau}(\mathcal{F}^s(X_{\tau}(x)))
&=&\bigcup_{n=1}^{\infty} X_{-n}(\mathcal{F}_{\tau}(X_n(x)))\\
&=&\bigcup_{n=1}^{\infty} X_{-n}(\mathcal{F}_{\tau}(X_n(x)) \cap B_{\delta/4}(X_n(x)))\\
&=&\bigcup_{n=1}^{\infty} X_{-n}(\mathcal{F}_{0}(X_n(x)) \cap B_{\delta/4}(X_n(x)))\\
&=&\bigcup_{n=1}^{\infty} X_{-n}(\mathcal{F}_{0}(X_n(x)))
=\mathcal{F}^s(x),
\end{eqnarray*}
and so $\mathcal{F}^s(X_{\tau}(x))=X_{\tau}(\mathcal{F}^s(x))$, which implies (1).

Now we prove (2).
Let $x,y\in \Lambda$ satisfy that $z\in \mathcal{F}^s(x)\cap \mathcal{F}^s(y)$ for some $z\in M$. 
By the same argument as for (\ref{FstB}), we can prove that
if 
\begin{equation}\label{Xtxyz}
\max\{d(X_t(x), X_t(z)), d(X_t(y), X_t(z))\} \le \delta/16
\end{equation}
for $t\ge 0$, then  
\begin{equation*}
\mathcal{F}^s_{loc}(x) \cap B_{\delta/4}(x)=\mathcal{F}^s_{loc}(y) \cap B_{\delta/4}(x),
\end{equation*}
were $\delta$ is as above.
This means that $\mathcal{F}^s(x)=\mathcal{F}^s(y)$.
In general, the condition (\ref{Xtxyz}) holds for sufficiently large $T$, 
because
$$\max\{d(X_t(x), X_t(z)),d(X_t(y), X_t(z))\} \to 0$$
as $n\to \infty$.
By using (1) we have $X_T(\mathcal{F}^s(x)) = \mathcal{F}^s(X_T(x))= \mathcal{F}^s(X_T(y))= X_T(\mathcal{F}^s(y))$, which gives the desired equality.
\end{proof}

\begin{lemma}
\label{perstable}
Let $\Lambda$ be a partially hyperbolic attractor with splitting $T_{\Lambda}M=E^s \oplus E^c$.
For a hyperbolic critical element $p\in \Lambda$ with $\dim W^{ss}(p)=\dim E^s$, we have $\mathcal{F}^s(p)=W^{ss}(p)$.
Moreover, for $x\in W^{ss}(p)\cap \Lambda$, we have $\mathcal{F}^s(x)=W^{ss}(p)$.
\end{lemma}

The proof of this lemma is similar to that of Lemma \ref{invariance} and for that reason we shall omit it.
For $z\in \Lambda$ and $\mu>0$ we set
$$
\mathcal{F}^s_{\mu}(z):=\{w\in \mathcal{F}^s(z):\rho^s(z,w)\le\mu\},
$$
where $\rho^s$ is the distance in $\mathcal{F}^s(z)$ induced by the
Riemannian metric. 
By Lemma \ref{perstable} we have
\begin{equation}\label{FsWssm}
\mathcal{F}^s_{\mu}(p)=W^{ss}_{\mu}(p)
\end{equation}
for a hyperbolic critical element $p\in \Lambda$ with $\dim W^{ss}(p)=\dim E^s$.

In the next proposition, the time-continuous version of \cite[Proposition~3]{BDT},
we recall some results relating some
shadowing properties with the location of the shadowing point in stable disks.
First we introduce a notation. Recall that $W^{s}(p)=\bigcup_{t\in \mathbb R} X_t(W^{ss}(p))$.
For $x\in W^{s}(p)$  and $\eta>0$ we will consider the local stable disk around $x$ in $W^{s}(p)$ given by 
$$
\gamma_{\eta}^s(x):=\{z\in W^{s}(p):d^s(x,z)\le\eta\}
$$
where $d^s$ is the distance in $W^{s}(p)$ induced by the Riemannian metric.

\begin{proposition}
\label{hairu2}
Let $\Lambda$ be a partially hyperbolic attractor with splitting $T_{\Lambda}M=E^s \oplus E^c$. For a hyperbolic critical element $p\in \Lambda$ with $\dim W^{ss}(p)=\dim E^s$,
there are $\varepsilon_1>0$ and $L>0$ such that for any $\varepsilon\in(0,\varepsilon_1)$
the following holds: if $x\in W^{ss}(p)\cap \Lambda$, $z\in M$ and $d(X_{t}(z),X_{t}(x))\le\varepsilon$ for any $t>0$
then  $z\in\gamma^s_{L\varepsilon}(x)$.
\end{proposition}

\begin{proof}
We consider only the case when $p$ is periodic since the singularity case can be shown similarly.
Put $\kappa=\min \{ \| X(X_t(p)) \| : t\in {\mathbb R}\}$ and note that $\kappa >0$.
Then we can take $t_0>0$ such that 
\begin{equation}\label{l}
d(X_t(p), X_s(p))\ge \kappa |t-s|/2
\end{equation}
for $|t-s|\le t_0$.

We claim that there exists $0<\mu\le t_0$ such that if $x\in \Lambda$ and 
{$y\in \mathcal{F}^s_{\mu}(x)$}, then 
\begin{equation}\label{dss}
\rho^{s}(x,y)\le 2d(x,y).
\end{equation}
Indeed, by the fact that $\| D_0 \exp_x \| =1$ and by Lemma \ref{s-disk} (1), for $0<\nu<\sqrt{2}-1$, there exists $\mu >0$ such that if $x\in \Lambda$, then
\begin{enumerate}
\item[$(1)$] $\| D_v \exp_x \| < 1+\nu$ for $v\in T_xM \cap \{ \|v \|<\mu\}$; 
\item[$(2)$] there exists a $C^1$ map $\psi_x:{E}_x^s \cap \{ \| v \| <\mu\}\to (E_x^s)^{\perp}$ such that
\begin{eqnarray*}
&&\exp_x^{-1}\{ \mathcal{F}^s_{\mu}(x) \} \subset \{ v + \psi_x(v) : v\in E^s_x, \| v \| < \mu\} \text{ and}\\
&&\| D_v\psi_x \|\le \nu \text{ for } v\in E^s_x \cap \{ \| v \| <\mu\}.
\end{eqnarray*}
Here $(E_x^s)^{\perp}$ is the orthogonal complement of ${E}_x^s$.
\end{enumerate}
Since, for $y\in \mathcal{F}^s_{\mu}(x)$, we can take $v\in {E}_x^s \cap \{ \| v \| <\mu\}$ such that $y=\exp_x (v + \psi_x (v))$, we have
$$\rho^s(x,y)\le (1+\nu)^2 \| v \| \le (1+\nu)^2 \| v + \psi_x (v)\|<2d(x,y),$$
which proves the claim.

By Lemma \ref{stable} we can choose $0<2\varepsilon_0<\mu/4$
such that for {$x\in M$, if $d(X_t(x), X_{t}(p))\le 2\varepsilon_0$} for every $t\ge 0$, then 
\begin{equation}\label{Wsmu}
x\in W^s_{\mu/4}(p)=\bigcup_{|t|\le \mu/4} X_t(W^{ss}_{\mu/4}(p)).
\end{equation}

Put $K=\max \{ \| X(x) \| : x\in M\}$ and $L_0=1+ 2 K /\kappa$.
Let {$0<\varepsilon < \varepsilon_1=\min \{ \varepsilon_0, \mu/4L_0\}$}.
Since $x\in W^{ss}(p)$, there is a sufficiently large $T>0$ such that 
\begin{equation*}
X_T(p)=p,~~
d(X_{t+T}(x), X_{t}(p))\le \varepsilon_0
\end{equation*}
for all $t\ge 0$. 
By the definition of $W^{ss}(p)$, (\ref{FsWssm}) and (\ref{Wsmu}) we have 
\begin{equation}\label{XTxp}
X_T(x)\in W^{ss}_{\mu/4}(p)=\mathcal{F}^s_{\mu/4}(p).
\end{equation}
By the assumption of $z$, we have 
\begin{equation*}
d(X_{t+T}(z), X_{t}(p))\le 
d(X_{t+T}(z), X_{t+T}(x)) + d(X_{t+T}(x), X_{t}(p))\le 2\varepsilon_0
\end{equation*}
for $t\ge 0$. By (\ref{FsWssm}) and (\ref{Wsmu}) we can find $t_1$ with $|t_1|\le \mu/4$ such that
\begin{equation}\label{XTzp}
X_{T}(z)\in X_{t_1}(W^{ss}_{\mu/4}(p))=X_{t_1}(\mathcal{F}^s_{\mu/4}(p)).
\end{equation}
Combining (\ref{XTxp}) and (\ref{XTzp}) we have
\begin{equation}\label{XTt1z}
X_{T-t_1}(z)\in \mathcal{F}^s_{\mu/2}(X_{T}(x)).
\end{equation}

Since $x\in W^{ss}(p)$, we have $d(X_{t}(x),X_{t}(p))\to  0$ $(t\to \infty)$.
Put $K_0=\max \{ \| D_xX_{t_1} \| : x\in M \}$. 
By (\ref{XTzp}) we have
\begin{eqnarray*}
d(X_{t}(z),X_{t+t_1}(p))&\le& K_0 d(X_{t-t_1}(z),X_{t}(p)) \to 0~~(t\to \infty).
\end{eqnarray*}
Thus it follows from (\ref{l}) that
\begin{eqnarray*}
\varepsilon &\ge & d(X_t(x), X_t(z)) \\
&\ge& d(X_t(p), X_{t+t_1}(p)) - d(X_t(p), X_{t}(x)) - d(X_{t+t_1}(p), X_{t}(z))\\ 
&\ge& \kappa |t_1|/2 - d(X_t(p), X_{t}(x)) - d(X_{t+t_1}(p), X_{t}(z)) \\
&\to& \kappa |t_1|/2 ~~(t\to \infty),
\end{eqnarray*}
which means that $|t_1|\le 2\varepsilon/\kappa$.
Recall $L_0=1+ 2 K /\kappa$. We have
\begin{eqnarray}
d(X_t(x), X_{t-t_1}(z)) &\le & d(X_t(x), X_t(z))+ d(X_t(z), X_{t-t_1}(z)) \nonumber \\
&\le& \varepsilon + K |t_1|\nonumber \\ 
&\le& (1+ 2 K /\kappa) \varepsilon =L_0 \varepsilon \label{1+2K/k}
\end{eqnarray}
for $t\ge 0$.
Now, take a small $t_2>0$ such that 
\begin{eqnarray*}
K_1=\max \{ \| D_xX_{-t} \| : x\in M, 0\le t \le t_2 \}\le 2. 
\end{eqnarray*}
Put $I=\{ t\in [0,\infty) : \rho^{s}(X_{t}(x),X_{t-t_1}(z))\le 2L_0\varepsilon \}$ and $t_0=\inf I$. 
By (\ref{dss}), (\ref{XTt1z}) and (\ref{1+2K/k}) we have $T\in I$. Assume that $t_0>0$. 
Since
\begin{eqnarray*}
\rho^{s}(X_{t_0-t_2}(x), X_{t_0-t_1-t_2}(z)) &\le & K_1 \rho^{s}(X_{t_0}(x), X_{t_0-t_1}(z)) 
\le 4L_0 \varepsilon \le \mu,
\end{eqnarray*}
by (\ref{dss}) and (\ref{1+2K/k}) we have $t_0-t_2\in I$, which is a contradiction. Thus $0=\inf I$. Therefore 
\begin{eqnarray*}
d^s(x,z)&\le& \rho^{s}(x, X_{-t_1}(z))+ d^s(X_{-t_1}(z),z)\\
&\le& 2L_0 \varepsilon + K|t_1| \le (2L_0+ 2K/\kappa) \varepsilon\le 3L_0\varepsilon.
\end{eqnarray*}
The proposition follows taking $L=3L_0$.
\end{proof}

\color{black}

\begin{lemma}
\label{holonomy}
Let $p\in \Lambda$ be a hyperbolic critical element with $\dim W^{ss}(p)=\dim E^s_p$ and $U$ be a subdisk of $\mathcal{W}^u(p)$ with $\overline{U} \subset \mathcal{W}^u(p)$. Then there exists $\mu>0$ such that $\mathcal{A}(U):=\bigcup_{z\in U}\mathcal{F}^s_{\mu}(z)$ is homeomorphic to $U\times [-\mu,\mu]^{\dim E^s}$.
\end{lemma}

\begin{proof}
Let $U$ be a subdisk of $\mathcal{W}^u(p)$ with $\overline{U} \subset \mathcal{W}^u(p)$.
It follows from Lemma~\ref{unstable} that $U\subset \Lambda$.
For $z\in U$ and $\mu>0$, we set $E^s_z(\mu)=E^s_z \cap \{ \| v \| \le \mu \}$.
Since $T_z \mathcal{F}^s(z)=E^s_z$, 
$\mathcal{F}^s_{\mu}(z)$ is the image of $E^s_z(\mu)$ under the exponential map $\varphi_z$ of $\mathcal{F}^s(z)$.
For every $z\in U$, $E_z^s(\mu)$ can be identified with $E_p^s(\mu)$ by parallel transport.
So we can obtain a surjective continuous map 
$$U \times E^s_p(\mu) \ni (z, v) \mapsto h(z,v):=\varphi_z(v)\in \cA(U).$$

Since $T_z \mathcal{F}^s(z)=E^s_z$ and $T_zW^u(p) \subset E^c_z$ for $z\in U$ (Lemma \ref{TWu}) and since $\overline{U} \subset \mathcal{W}^u(p)$, 
we can take $\mu>0$ small enough such that $\mathcal{F}^s_{\mu}(z) \cap \overline{U}=\{ z\}$ for $z\in \overline{U}$.
Then, by Lemma~\ref{invariance} (2), we have that the above map $h$ is injective.
Since $h$ can be defined on some domain slightly larger than $U \times E^s_p(\mu)$, by the Brouwer's invariance of domain theorem we can show that $h$ is a homeomorphism.
Since $E^s_p(r)$ is a disk of radius $r$ of dimension $\dim E^s$, $E^s_p(r)$ is homeomorphic to the set $[-1,1]^{\dim E^s}$. This proves the lemma.
\end{proof}

Let $\cA(U)$ be as in Lemma \ref{holonomy} and define $\pi^s\colon \cA(U)\to U$
by $\pi^s(x)=z$ if $x\in \mathcal{F}^s_{\mu}(z)$. By Lemma \ref{holonomy}, $\pi^s$ is well-defined. 
Since $\{ \mathcal{F}^s_{\mu}(z) \}_{z\in U}$ is continuous, so is $\pi^s$.

\begin{lemma}\label{finite}
Let $p, q\in \Lambda$ be hyperbolic critical elements with $\dim W^{ss}(p)=\dim E^s<\dim W^{ss}(q)$ and let $\pi^s:\cA(U)\to U$ be as above. Then $\pi^s(X_T(W^u_{\mu}(q)))$ is contained in a finite union of (topological) disks of  $\dim W^u_{\mu}(q)$.
\end{lemma}

\begin{proof}
Without loss of generality we may assume that $\cA(U')$ can be defined for some open disk $U'$ in $W^{u}(p)$ satisfying $\bar{U} \subset U'$.
It follows from Lemma \ref{cdisk} that for $x\in \cA(\bar{U}) \cap X_T(W^u_{\mu}(q))(\subset \cA(\bar{U}) \cap W^u(q))$, there exist a neighborhood $U(x)$ of $x$ and a $C^1$-disk $D(x)$ such that
\begin{eqnarray*}
x\in W^u_{x}(q) \subset D(x) \subset \cA(U') \text{ and } T_xD(x)=E^c_x,
\end{eqnarray*}
where $W^u_x(q)$ denotes the connected component of $W^u(q) \cap U(x)$ containing $x$.
Since $\pi^s: D(x) \to \pi^s(D(x))$ is a homeomorphism, $\pi^s(W^u_{x}(q))$ is a topological disk of $\dim W^u(q)$.

Consider an open cover ${\mathcal D}=\{ D(x) \}$ of $\cA(\bar{U}) \cap X_T(W^u_{\mu}(q))$. By the compactness we can take a finite subcover ${\mathcal B}=\{ D(x_i) \}$ of ${\mathcal D}$. Then $\pi^s(X_T(W^u_{\mu}(q))\subset \cup_i \pi^s(W^u_{x_i}(q))$, which proves the Lemma.
\end{proof}

\color{black}

\section{Proof of Theorem \ref{thm:main2}}\label{proofs}

The aim of this section is to prove our main result. 
Let $X\in \mathfrak{X}^1(M)$ be a $C^1$ vector field so that the flow $(X_t)_{t\in\mathbb R}$ admits a partially
hyperbolic attractor $\Lambda$ with splitting $T_\Lambda M= E^s\oplus E^c$ and assume that there are two 
hyperbolic critical elements  $p$ and $q$ such that ${\rm dim}\;E^s={\rm dim}\;W^{ss}(p)<{\rm dim}\;W^{ss}(q)$. 
The key idea involved in the proof of Theorem \ref{thm:main2} is to notice that the {weak} specification property 
implies that strong stable and unstable manifolds intersect in a strong way. 
It is well known that the {weak} specification property implies that there exists a time $T>0$ depending only on the size of the local 
stable/unstable manifolds so that the union of the image of the local strong unstable manifold 
of a hyperbolic critical element by the maps $(X_t)_{t\in [0,T]}$ must intersect the local strong stable manifolds of any other hyperbolic critical element (see e.g.~\cite{SSY} and \cite{ AST} for statements in the discrete-time and continuous-time settings, respectively). By {weak} specification property this non-empty intersection property condition should hold not only at hyperbolic critical elements but whenever two points admit stable and unstable manifolds. Due to the assumption of partial hyperbolicity and different indexes, here we can choose the hyperbolic critical elements properly to prove that there exists a uniform size 
and a point on the strong unstable manifold of a critical element whose image of its local unstable disk does not
intersect the local stable disk of the other hyperbolic critical element  (c.f. statement of the Sublemma below).

Since singularities and periodic orbits have different structure at local coordinates, which is reflected by the fact that 
stable/unstable and strong stable/unstable manifolds coincide at singularities while this does not happen at periodic points,  
it is natural to subdivide the proof in four cases, corresponding to the ones where the two hyperbolic critical elements 
$p,q$ are either singularities/periodic orbits and also on the dimension of their strong stable manifold.

{To reach a contradiction
we assume} that $\Lambda$ has the weak specification property.
Then $(X_t)_{t\in \mathbb R}$ is topologically mixing (c.f. \cite[Lemma~3.1]{AST}) and it admits neither attracting nor 
repelling critical elements. 
There are four cases to consider depending on whether $p$ and $q$
are singularities or periodic orbits.

\vspace{.15cm}
\noindent \textbf{First case:} \textit{$p$ and $q$ are singularities} 
\vspace{.15cm}

In this case we remark $W^u(p)=W^{uu}(p)$ and $W^u(q)=W^{uu}(q)$. Take an open disk $D_0 = W_\mu^{u}(p) \subset \Lambda$ with respect to the induced topology on $W^{uu}(p)$, which is transverse to the local stable foliation through points of $D_0$ (see Lemma \ref{TWu}). 
It follows from Lemma~\ref{holonomy} that if $\mu>0$ is small then
$
\mathcal{A}(D_0):=\bigcup_{z\in D_0}\mathcal{F}^s_{\mu}(z)
$
is homeomorphic to $D_0\times [-\mu,\mu]^{\dim\ E^s}$,
where we set
$$
\mathcal{F}^s_{\mu}(z):=\{w\in \mathcal{F}^s(z):\rho^s(z,w)\le\mu\}
$$
and $\rho^s$ is the distance in $\mathcal{F}^s(z)$ induced by the Riemannian metric. 
By the choice of $D_0$ we have 
$\dim D_0=\dim W^{uu}(p)=\dim E^c_p.$
Let $\epsilon_1>0$ and $L>0$ (depending on $p$) be given by Proposition~\ref{hairu2}. 
We claim the following:

\vspace{.15cm}
\noindent {\bf Sublemma:} 
There are $\mu>0$,
$\epsilon\in(0,\epsilon_1)$ with $\epsilon<\mu/L$ and $x\in W^{ss}(p)$ so that if $T=T(\epsilon)$ 
is given by the weak specification property then
$
X_{-T}(\gamma_{\mu}^s(x))\cap W_{\mu}^u(q)=\emptyset.
$

\begin{proof}[Proof of the Sublemma]
Take $\mu>0$ so that $\mathcal{A}(D_0):=\bigcup_{z\in D_0}\mathcal{F}^s_{\mu}(z)$ is homeomorphic to $D_0\times [-\mu,\mu]^{\dim E^s}$.
Set $\epsilon:=\min\{\mu/5,\epsilon_1/2\}$ and let $T(\epsilon)$ be as above.
By Lemma \ref{finite}, the projection $\pi^s(X_{T}(W^u_\mu(q)))$ along the stable holonomy is contained in a finite union of disks of ${\rm dim}\;W^{u}_{\mu}(q)={\rm dim}\;W^{uu}(q)<{\rm dim}\;D_0$ (see Figure~\ref{fig:1} below). Here the map $\pi^s\colon \cA(D_0)\to D_0$ is defined by $\pi^s(x)=z$ if $x\in \mathcal{F}^s_{\mu}(z)$.
\begin{figure}[htbp]
 \begin{center}
  \includegraphics[width=100mm]{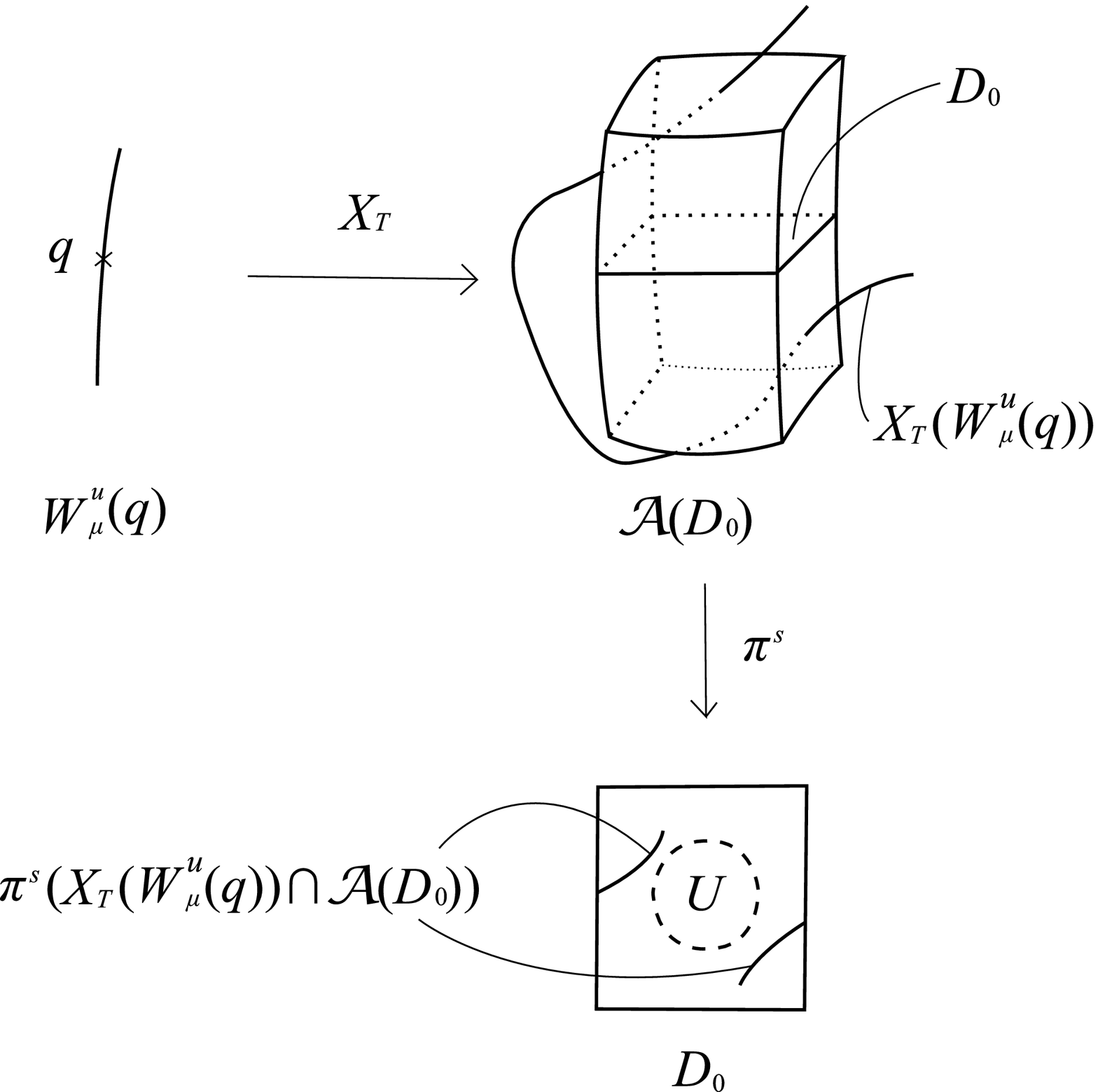}
 \end{center}
 \caption{}
 \label{fig:1}
\end{figure}
Since the complement of $\pi^s(X_{T}(W^u_\mu(q)))$ is open in $D_0$, 
there exists an open disk $U\subset D_0$ so that 
$\mathcal A(U) \cap X_T(W^u_\mu(q)) =\emptyset$.
Since $U\subset \Lambda$ and the stable manifold $W^s(p)$ is dense in $\Lambda$ (Lemma \ref{dense}), then there exists $x\in D_0 \cap \cA(U) \cap W^s(p)$ with 
$\cF_\mu^s(x) \subset \mathcal A(U)$ disjoint from $X_T(W_\mu^u(q))$.
Hence $X_{-T}(\gamma_\mu^s(x)) \cap W_\mu^u(q) =\emptyset $
\end{proof}

We proceed with the proof of 
Theorem \ref{thm:main2} in the first case.
 On the one hand, 
by the {sublemma} there exist $\mu>0$,  $0<\epsilon< \min\{\mu/L, \epsilon_1\}$ and 
$x\in W^s(p)$  so that $X_{-T}(\gamma_{\mu}^s(x))$ does not intersect $W_{\mu}^{u}(q)$,
where $T=T(\epsilon)>0$ is given by the specification property. 

On the other hand, for the singularity $q$ and $x\in W^{ss}(p)$ given by the previous {sublemma}, 
by compactness of $\Lambda$ and the specification property 
there exists $z\in \Lambda$ such that 
$d(X_t(z),X_t(x))\le\epsilon$
and 
$d(X_{-t}(X_{-T}(z)),X_{-t}(q))\le\epsilon$ 
for all $t\ge 0$.
Since $\epsilon\in (0,\epsilon_1)$, Proposition~\ref{hairu2} guarantees that $z\in X_{-T}(\gamma_{L\epsilon}^s(x))\cap W_{L\epsilon}^u(q)$, which is a contradiction
since $L\epsilon<\mu$. This finishes the proof of Theorem~\ref{thm:main2} in this first case.

\vspace{.35cm}
\noindent \textbf{Second case:} \textit{$p$ and $q$ are periodic orbits} 
\vspace{.15cm}

The strategy is again to deduce a contradiction by assuming the specification property. 
Given $\mu>0$ consider the disk $D_0=W^u_{\mu}(p):= \cup_{|t|\le \mu} X_t(W^{uu}_\mu(p))$ containing $p$ and the strong stable holonomy $\pi^s$ defined in
$
\mathcal{A}(D_0):=\bigcup_{z\in D_0}\mathcal{F}^s_{\mu}(z).
$
Take $\mu>0$ so that $4\mu$ is smaller than the prime periods of $p$ and $q$, and $\mathcal{A}(D_0)$ is homeomorphic to $D_0\times [-\mu,\mu]^{\dim E^s}$ (Lemma~\ref{holonomy}).

Let $\pi:D_0\to W^{uu}_\mu(p)$ be the projection along the orbit, i.e. if $x=X_t(z)\in D_0$ for some $|t|\le \mu$ and $z\in W^{uu}_\mu(p)$, then $\pi(x)=z$.
By definition we have $\pi(\pi^s(\mathcal{A}(D_0)))=W^{uu}_\mu(p)$, which means that $\mathcal{A}(D_0)=\bigcup_{z\in W^{uu}_\mu(p)} (\pi^s)^{-1}( \pi^{-1}(\{ z \})).
$
By Lemma \ref{invariance}, if $X_t(x)\in \mathcal{A}(D_0)$ for $t\in [0,t_0]$, then 
$$X_t \circ \pi^s(x)=\pi^s \circ X_t(x) ~~(t\in [0,t_0]).$$
This implies that for $z\in W^{uu}_\mu(p)$, we have
$$
(\pi^s)^{-1}( \pi^{-1}(\{ z \}))=\bigcup_{|t|\le \mu} X_t(\mathcal{F}^s_{\mu}(z)).
$$
Thus $\{ (\pi^s)^{-1}( \pi^{-1}(\{ z \})) \}_{z\in W^{uu}_\mu(p)}$ is a $C^1$-continuous family of $(1+\dim E^s)$-dimensional disks in $\mathcal{A}(D_0)$. 
By Lemma \ref{perstable}, if $(\pi^s)^{-1}( \pi^{-1}(\{ z \})) \cap W^{ss}(p)\not= \emptyset$, then $$(\pi^s)^{-1}( \pi^{-1}(\{ z \}))\subset W^s(p).$$
We can take $\tau>0$ such that $(\pi^s)^{-1}( \pi^{-1}(\{ z \}))$ contains a $(1+\dim E^s)$-dimensional ball centered at $z$ with radius $\tau$ for $z\in W^{uu}_\mu(p)$.

Let $\epsilon_1>0$ and $L>0$ (depending on $p$) be as in Proposition \ref{hairu2}.
Then 
(up to time reversal in Proposition \ref{hairu2})
we can choose $0<\epsilon\le \min \{ \epsilon_1, \tau/2L \}$ such that for $x\in \Lambda$, if $d(X_{-t}(x), X_{-t}(q))\le \epsilon$ for $t\ge 0$, then 
\begin{eqnarray}\label{Wumu}
x\in W^u_{\mu}(q)=\bigcup_{|t|\le \mu} X_t(W^{uu}_{\mu}(q)).
\end{eqnarray}
By definition, $W^u_{\mu}(q)$ is foliated by pieces of orbits of points in $W^{uu}_{\mu}(q)$ and so $\dim W^{u}_{\mu}(q)=1+\dim W^{uu}_{\mu}(q)=1+\dim F^u_q$.
Let $T=T(\epsilon)>0$ be as in the definition of the specification property.
Then $X_T(W_\mu^u(q))$ is also foliated by pieces of orbits.

On the one hand, since $X_T(W_\mu^u(q))$ is a $(1+\dim F^u_q)$-dimensional submanifold, 
by Lemma \ref{finite} we have that $\pi^s (X_T(W_\mu^u(q)) \cap {\mathcal A}(D_0))$ is contained in a finite union of compact disks of dimension $1+\dim F^u_q$. Since $X_t \circ \pi^s=\pi^s \circ X_t$ in ${\mathcal A}(D_0)$ (see Lemma \ref{invariance}), such compact disks are also foliated by pieces of orbits.

Let $\pi:D_0\to W^{uu}_\mu(p)$ be as above. Since $\pi$ reduces each piece of orbit to one point, $(\pi \circ \pi^s) (X_T(W_\mu^u(q))\cap {\mathcal A}(D_0))$ is contained in a finite union of compact disks of $\dim F^u_q<\dim F^u_p=\dim W^{uu}_{\mu}(p)$.
Since the complement of $(\pi \circ \pi^s) (X_T(W_\mu^u(q))\cap {\mathcal A}(D_0))$ is open
and dense in $W^{uu}_{\mu}(p)$, there exists an open disk $U\subset W^{uu}_{\mu}(p)$ so that 
$$U \cap (\pi \circ \pi^s) (X_T(W_\mu^u(q))\cap {\mathcal A}(D_0))=\emptyset,$$
which means that
$${\mathcal A}(\pi^{-1}(U)) \cap X_T(W^u_\mu(q))=(\pi^s)^{-1}(\pi^{-1}(U)) \cap X_T(W^u_\mu(q)) =\emptyset,$$
as illustrated by Figure~\ref{fig:2} below.

\begin{figure}[htbp]
 \begin{center}
  \includegraphics[width=100mm]{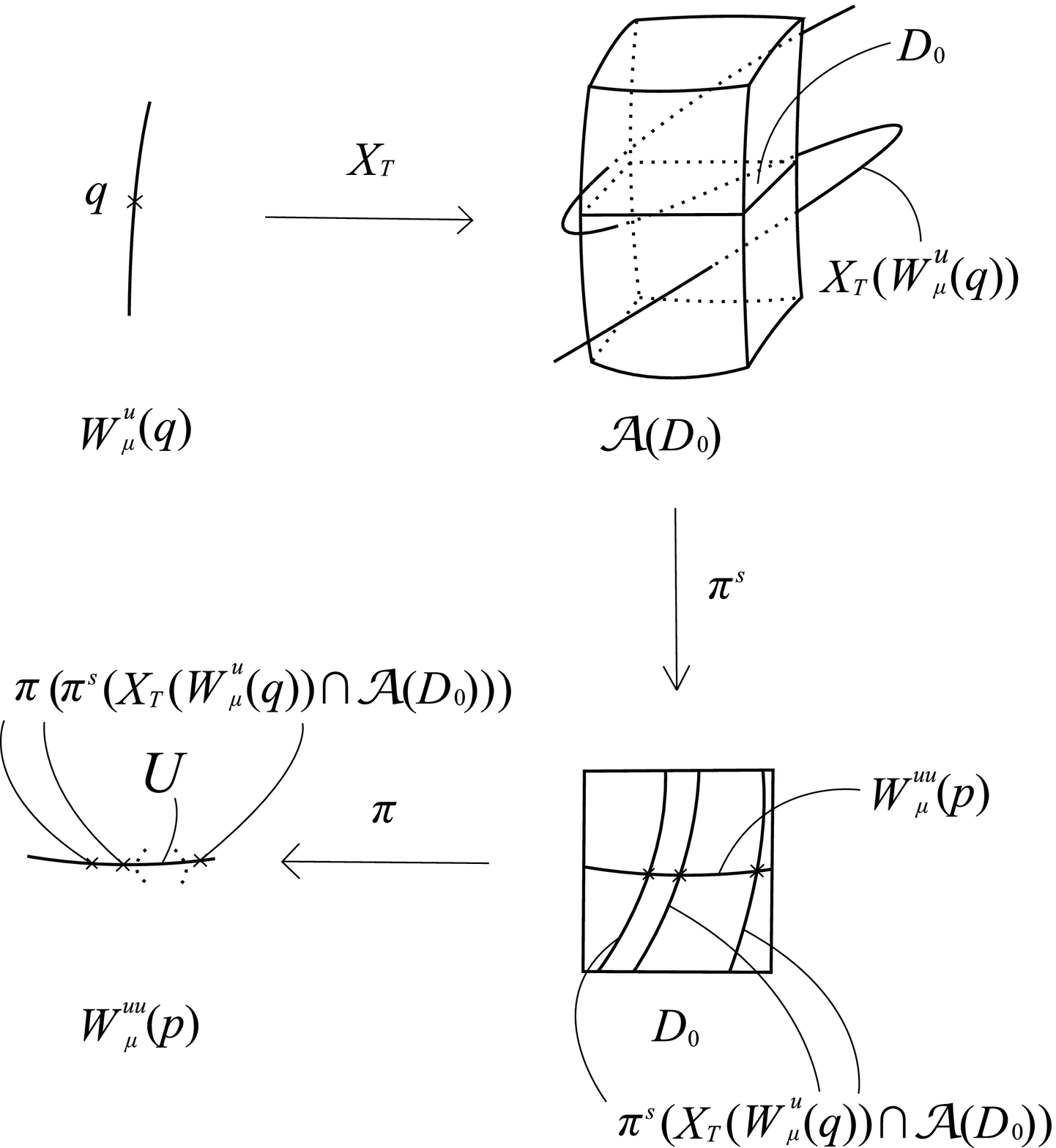}
 \end{center}
 \caption{}
 \label{fig:2}
\end{figure}

On the other hand, since $W^{ss}(p)$ is dense in $\Lambda$ (Lemma \ref{dense}), we have ${\mathcal A}(\pi^{-1}(U)) \cap W^{ss}(p)\not= \emptyset$. Furthermore, we can choose a point $w\in \pi^{-1}(U) \cap W^{ss}(p) (\subset D_0)$ which is so close to $U(\subset W^{uu}_{\mu}(p))$ that $$\gamma^s_{L\epsilon}(w):=\{z\in W^{s}(p):d^s(w,z)\le L\epsilon\}\subset {\mathcal A}(\pi^{-1}(U))$$
since $\epsilon <\tau/2L$.

By the specification property, there exists $y\in\Lambda$ so that  
$
d(X_{-t}(y), X_{-t}(q)) \le \epsilon
$
and
$d(X_{t}(X_T(y)), X_{t}(w)) \le \epsilon	$
for all $t\ge 0$. By (\ref{Wumu}) and Proposition \ref{hairu2}, we have
$y \in W_{\mu}^{u}(q)$ and $X_T(y) \in \gamma^s_{L\epsilon}(w)$. Since $\gamma^s_{L\epsilon}(w) \subset {\mathcal A}(\pi^{-1}(U))$,
we have ${\mathcal A}(\pi^{-1}(U)) \cap X_T(W^u_\mu(q)) \not=\emptyset$, which is a contradiction.

\vspace{.35cm}
\noindent \textbf{Third case:} \textit{$p$ is a singularity and $q$ is a periodic orbit} 
\vspace{.15cm}

The strategy is again to deduce a contradiction by assuming the specification property. 
Let us observe that in this setting
$$
\dim W^{u}(q)= 1+\dim F^u_q = n-\dim F^s_q
	< n-\dim F^s_p 
	= \dim W^{u}(p).
$$
Thus, we can apply the argument proving that the complement of the set $\pi^s(X_T(W_\mu^u(q)))$
(here $\pi^s$ denotes again the strong strong stable holonomy map in a neighborhood of $p$ on 
a disk $D_0 \subset W^u(p)$) contains open sets $U \subset D_0$, as well as the 
proof that this property prevents specification.

\begin{remark}
Let us mention that simpler third case is only relevant in the dimension larger than or equal to 4. In fact, if $\dim M=3$ then  necessarily $\dim F^u_q=\dim F^s_q=1$ and $\dim F^s_p<\dim F^s_q$ leads to a 
contradiction to the fact that $F^s$ is non-trivial.
\end{remark}

\vspace{.35cm}
\noindent \textbf{Fourth case:} \textit{$q$ is a singularity and $p$ is a periodic orbit} 
\vspace{.15cm}

To finish the proof of Theorem~\ref{thm:main2} it remains to deal with the case that $q$ is a singularity 
and $p$ is a periodic orbit. Now the relations
$
\dim M = \dim F^s_q+ \dim F^u_q
$
and also
$
\dim M = \dim F^s_p+ \dim F^u_p+1
$
together with $\dim F^s_p \le \dim F^s_q-1$ yield that $\dim F^u_q \le \dim F^u_p$. 
If the strict inequality holds we can proceed as in the third case. Otherwise, the difficulty occurs if $\dim F^u_q = \dim F^u_p$. 
Nevertheless, $\pi^s$ is a projection defined in a neighborhood of $p$ onto the local weak unstable manifold $W^u(p)$, and $\dim W^u(p)=1+\dim F^u_p>\dim F^u_q$. The argument works as before:
taking $D_0=W^u_\mu(p)$ it follows that $\pi^s (X_T(W^{uu}_{\mu}(q))\cap {\mathcal A}(D_0))$ is contained in a finite union of disks of dimension $\dim F^u_q<1+\dim F^u_p=\dim W^u(p)$. 

Since 
$
X_t(x)\in W^{uu}(q)=W^{u}(q)
$
for $x\in W^{uu}_{\mu}(q)$ and $t\in {\mathbb R}$,
$W^{uu}_{\mu}(q)\setminus \{ q\}$ is foliated by pieces of orbits.
Thus $\pi^s (X_T(W_\mu^{uu}(q)) \cap {\mathcal A}(D_0))$ is contained in a finite union of disks which are foliated by pieces of orbits.
Let $\pi:D_0\to W^{uu}_\mu(p)$ be the projection along the orbit.
Then the dimension of $(\pi \circ \pi^s) (X_T(W_\mu^{uu}(q))\cap {\mathcal A}(D_0))$ is less than $\dim F^u_q-1<\dim F^u_p=\dim W^{uu}_{\mu}(p)$.
Since the complement of $(\pi \circ \pi^s) (X_T(W_\mu^{uu}(q))\cap {\mathcal A}(D_0))$ is dense and open in $W^{uu}_{\mu}(p)$, there exists an open disk $U\subset W^{uu}_{\mu}(p)$ so that 
$$U \cap (\pi \circ \pi^s) (X_T(W_\mu^u(q))\cap {\mathcal A}(D_0))=\emptyset,$$
which means that ${\mathcal A}(\pi^{-1}(U)) \cap X_T(W^u_\mu(q)) =\emptyset$. The proof follows the 
same lines as above.

\begin{remark}
In fact, in the case of the geometric Lorenz attractor in dimension 3 necessarily $\dim F^u_p=\dim F^s_p=1$
and for the singularity $\dim F^u_p=1$ and consequently $\dim F^u_p=\dim F^u_q$ leading to the 
fourth situation.
\end{remark}

\section{Proof of the Corollaries}\label{proofs2}

\subsection{Proof of Corollary~\ref{cor1}}

Let $p,q$ be hyperbolic critical elements so that ${\rm ind}^s(p) \neq {\rm ind}^s(q)$. 
{Here we set ${\rm ind}^s(p) =\dim W^{ss}(p)$.}
Observe that due to {strong hyperbolicy} $p$ and $q$ are neither attractors nor repellers.
We shall prove that $X$ or $-X$ satisfies the conditions of  Theorem~\ref{thm:main2} and, consequently, 
 $X$ does not satisfy the weak specification property.
For simplicity, we assume that $\text{dim} M = 4$.

(i) If $p,q$ are both periodic points then necessarily {${\rm ind}^s(p)\in \{1,2\}$} or {${\rm ind}^s(q)\in \{1,2\}$.}
Without loss of generality assume that {${\rm ind}^s(p)\in \{1,2\}$.} 
If ${\rm ind}^s(p) =1$ then $X$ satisfies the conditions of 
Theorem~\ref{thm:main2}. If {${\rm ind}^s(p) =2$} then $-X$ satisfies the conditions of Theorem~\ref{thm:main2}.

(ii) If $p,q$ are both singularities then ${\rm ind}^s(p)$ and ${\rm ind}^s(q)$ cannot be simultaneously $2$. The 
argument is completely analogous to the previous case.

(iii) Assume that $p$ is a periodic point and $q$ is a singularity. If ${\rm ind}^s(p)=1=\dim E^s$,
then $X$ satisfies the assumptions of Theorem~\ref{thm:main2}, since 
${\rm ind}^s(q) \neq 1$. If ${\rm ind}^s(p)=2$ then 
${\rm ind}^u(p)=1$ and consider $-X$. This completes the proof of the corollary.

\subsection{Proof of Corollary~\ref{nh}}

Put ${\mathcal U}=\mathcal{RNTF}\cap\mathcal{SPHF}_3(M)$. Since $X\in {\mathcal U}$ is robustly transitive, $X$ has no singularity (see \cite{Vi03}).
We note that ${\mathcal U} \cap {\mathcal G}^1(M)= \emptyset$ where $\mathcal{G}^1(M)$ is the class of star-flows 
(i.e. flows such that all critical elements are hyperbolic  $C^1$-robustly).
Indeed, to reach a contradiction we assume that there exists $X\in {\mathcal U} \cap {\mathcal G}^1(M)$. 
In \cite{GW06} Gan and Wen showed that if $X\in {\mathcal G}^1(M)$ has no singularity, then the nonwandering set of $X$ is hyperbolic, which means that $X$ is Anosov. This contradicts the fact that $X$ is not hyperbolic.

Let $X\in {\mathcal U}$. Since $X\not\in {\mathcal G}^1(M)$, $X$ can be approximated by a flow $Y\in {\mathcal U}$ having a non-hyperbolic periodic orbit. 
By the proof of Theorem~4.3 in \cite{AST}, we can find $Z\in {\mathcal U}$ arbitrarily close to $Y$ and having two hyperbolic periodic orbits with different indices, which is a $C^1$-open condition. 
Thus Corollary \ref{c3} implies that $Z$ does not satisfy the weak specification property $C^1$-robustly.

\subsection{Proof of Corollary~\ref{dim4}}

Following \cite{Vi03}, given  $X\in \mathcal{RNTF}$ it follows that $X$ has no singularity and the linear 
Poincar\'e flow $P^t=\pi_{\mathcal N_{X_t(x)}} \circ DX_t(x): \mathcal N_x \to \mathcal N_{X_t(x)}$  admits a dominated splitting: for every $x\in M$ there exists a $DP^t$-invariant and continuous decomposition of the normal space
$\mathcal N_x=E_x\oplus F_x$ and constants $C>0$ and $0<\lambda<1$ so that 
$$
\| DP^t |E_x\| \; \| (DP^t |F_{X_t(x)})^{-1}\| \le C \lambda^t 
$$
for every $t\ge 0$. 

We now proceed to prove that the one-dimensional subbunddle is hyperbolic. Assume for simplicity
that $\dim E=1$ and $\dim F=2$. Note that a robustly transitive flow does not have repelling periodic 
orbits. We claim that there exists $\delta>0$ such that 
$
|\lambda_E(p)| \le (1-\delta)^T < 1 
$
for every periodic point $p$ of period $T$, where $\lambda_E(p)$ denotes the 
eigenvalues of $DP^T(p)\mid_{E_p}$ (as otherwise one could use the Franks' lemma for flows 
as in the proof of \cite[Lemma~{4.5}]{DPU} to create a repelling periodic orbit).  
The proof that $E$ is uniformly contracting follows the same lines as the 
proof of the stability conjecture using the ergodic closing lemma given by Wen
(c.f. Step~3 in ~\cite[page 347]{Wen}).

We put $E^c_x=\langle X(x) \rangle \oplus F_x(\subset T_xM)$ for $x\in M$. Here $\langle X(x) \rangle$ denotes the one dimensional subspace generated by $X(x)$. Then $E^c$ is a $(DX_t)_{t\in \mathbb R}$-invariant subbundle. Since $E$ is uniformly contracting, as in the proof of \cite[Theorem~{1.5}]{Sal}, we can define a $(DX_t)_{t\in \mathbb R}$-invariant continuous one-dimentional subbundle $E^s\subset \langle X \rangle \oplus E$ such that the splitting $E^s\oplus E^c$ is partially hyperbolic.

By \cite[Theorem~A]{GW06} every non-singular star-flow is Axiom A without cycles. Since 
$X$ generates a non-hyperbolic robustly transitive flow without singularities then $X\not\in {\mathcal G}^1(M)$
and, consequently, $X$ can be $C^1$-approximated by a flow $Z\in {\mathcal U}$ with
two hyperbolic periodic orbits with different indices (\cite[Theorem~4.3]{AST}), which is a $C^1$-open condition. 
This finishes the proof of the corollary.

\section*{Acknowledgements}
P.V. was partially supported by a CNPq-Brazil post-doctoral fellowship at Universidade do Porto and N.S. is partially supported by KAKENHI (15K04902).
P.V. and N.S. are grateful to the organizers of the conference \emph{ICM 2014 Satellite Conference on Dynamical Systems and Related 
Topics, Daejeon-Korea} for the hospitality in Korea, where part of this work was developed. The authors are 
greatly indebted to the anonymous referees for their criticism and an exhaustive list of suggestions that helped 
to improve significantly the article.



\begin{thebibliography}{99}


\bibitem{AP} A. Ara\'ujo and M. J. Pacifico, \textit{Three dimensional flows},
Ergebnisse der Mathematik und ihrer Grenzgebiete. 3. folge. A series of Modern
surveys in Mathematics 53, Springer-verlag Heidelberg, 2010.


\bibitem{ADK} N. Aoki, M. Dateyama and M. Komuro, \textit{Solenoidal automorphisms with specification},
Monatsh. Math., {\bf 93} (1982), 79--110.


\bibitem{Arbietoexp}
A. Arbieto.
\newblock On persistently positively expansive maps.
\newblock An. Acad. Brasil. Cienc. {\bf 82} (2010), no. 2, 263--266.


\bibitem{AM} A.~Arbieto and C.~A. Morales, \textit{A dichotomy for higher-dimensional flows}, Proc. Amer. Math. Soc. {\bf  141} (2013), 2817--2827.



\bibitem{AST} A.~Arbieto, L.~Senos and T.~Sodero, \textit{The specification property for flows from the 
robust and generic viewpoint}, J. Differential Equations {\bf  253} (2012), 1893--1909.

\bibitem{BDT} C. Bonatti, L. D\'iaz and G. Turcat, \textit{Pas de ``shadowing lemma"
pour des dynamiques partiellement hyperboliques}, C. R. Acad. Sci. Paris Ser. I Math., {\bf 330} (2000), 587-592.


\bibitem{BDV05}
C. Bonatti, L. D\'iaz and M. Viana.
\textit{ Dynamics beyond uniform hyperbolicity}, Springer, 2005.


\bibitem{Bonatti}
C. Bonatti.
\newblock Survey: Towards a global view of dynamical systems, for the C1-topology.
\newblock Ergodic Theory Dynam. Systems {\bf 31} (2011), no. 4, 959--993. 

\bibitem{BDP}C. Bonatti, L. D\'iaz and E. Pujals, \textit{A $C^1$-generic dichotomy for diffeomorphisms: weak forms of hyperbolicity or infinitely many sinks or sources}, Ann. of Math. (2) {\bf 158}
(2003), 355--418.



\bibitem{B} R. Bowen, \textit{Periodic points and measures for Axiom A diffeomorphisms}, Trans. Amer.
Math. Soc., {\bf 154} (1971), 377--397.

\bibitem{Bo72} R. Bowen, \textit{Periodic points for hyperbolic flows}, Amer. Journal Math.
Math. Soc., {\bf 94},1, (1972), 1--30.

\bibitem{Bo71} R.~Bowen, \textit{Entropy for group endomorphisms and homogeneous spaces},
Trans. Amer. Math. Soc., {\bf 153} (1971), 401--414.


\bibitem{CSY} S. Crovisier, M. Sambarino and D. Yang, \textit{Partial hyperbolicity and homoclinic tangencies},
arXiv:1103.0869.

\bibitem{DPU}L. D\'iaz, E. Pujals and R. Ures, \textit{Partial hyperbolicity and robust transitivity},
Acta Math. {\bf 183} (1999), 1--43.

\bibitem{DGS}
M. Denker, C. Grillenberger and K. Sigmund.
\newblock Ergodic theory on compact spaces.
\newblock \emph{Lecture Notes in Mathematics, Vol. 527}, . Springer-Verlag, Berlin-New York, 1976.

\bibitem{DPU} 
L. D\'iaz, E. Pujals and R. Ures,  Partial hyperbolicity and robust transitivity, 
 Acta Mathematica, 183, 1--43, (1999).

\bibitem{Doering} 
C. I. Doering. 
Persistently transitive vector fields on three-dimensional manifolds. Proceedings on Dynamical Systems and Bifurcation Theory 160 (1987), 59--89.


\bibitem{GW06} 
S.~Gan and L.~Wen, \textit{Non-singular star flows satisfy Axiom A and the no-cycle condition}, Invent. Math., {\bf 164}, (2006), 279--315.


\bibitem{G}
N.~Gourmelon, \textit{Adapted metrics for dominated splittings}, Ergod. Th. {\&} Dynam. Sys., {\bf 27} (2007), 1--11.


\bibitem{Gu}
J.~Guckenheimer, \textit{A strange strange attractor}, The Hopf bifurcation and its applications, Applied Mathematical Series 19 (1976), 368-381.

\bibitem{GuWi}
J.~Guckenheimer and R.~Williams, \textit{Structural stability of Lorenz attractors}, Inst. Hautes \'Etudes Sci. Publ. Math. {\bf 50} (1979), 59--72.



\bibitem{Hayashi}
S. Hayashi.
\newblock Connecting invariant manifolds and the solution of the $C^1$-stability 
	and $\Omega$-stability conjectures for flows.
\newblock   Ann. of Math. (2) {\bf 145} (1997), 81--137 and Ann. of Math. (2) {\bf 150} (1999), 353--356.


\bibitem{HPPS} M.~Hirsch, J.~Palis, C.~Pugh and M.~Shub, \textit{Neighborhoods of hyperbolic sets}, Invent. Math. {\bf 9} (1970) 121--134. 


\bibitem{HPS} M. W. Hirsch, C. C. Pugh and M. Shub, \textit{Invariant Manifolds}, Lecture Notes in Mathematics, {\bf 583}, Springer, Berlin, (1977).



\bibitem{KH}
A.~Katok, B.~Hasselblatt.
\newblock{ Introduction to the modern theory of dynamical systems}.
\newblock{\em {Encyclopedia of Mathematics and its Aplications}},  {1997}.



\bibitem{KS}
S.~Kiriki and T.~Soma,
\textit{Parameter-shifted shadowing property for geometric Lorenz attractors}, 
Trans. Amer. Math. Soc. {\bf 357} (2005), 1325-1339.


\bibitem{Ko} M. Komuro, Lorenz attractors do not have the pseudo-orbit tracing property, J. Math. Soc. Japan, 
37, (1985), 489--514.

\bibitem{L} D. A. Lind, \textit{Ergodic group automorphisms and specification}, Ergodic Theory
(Proc. Conf., Math. Forschungsinst., Oberwolfach, 1978), 93--104, Lecture Notes in Math.,
{\bf 729}, Springer-Verlag, Berlin, (1979).

\bibitem{Lop} A. L\'opez, {\it Existence of periodic orbits for sectional-{A}nosov flows}, Preprint ArXiv:1407.3471.


\bibitem{Mane0}
R. Ma\~n\'e. 
\newblock Contributions to the stability conjecture. 
\newblock  Topology {\bf 17} (1978), 383--396.



\bibitem{Mane}
R. Ma\~n\'e. 
\newblock A proof of the $C^1$-stability conjecture. 
\newblock Publ. Math. Inst. Hautes ﾉtudes Sci. {\bf 66} (1988), 161--210.


\bibitem{MM} Metzger, R. and Morales, {\it C. Sectional-hyperbolic systems}.  Ergodic Theory Dynam. Systems 28, no. 5, 2008, 1587-1597.

\bibitem{MPP}
C.~A. Morales, M.~J. Pacifico, and E.~R. Pujals.
\newblock {Robust transitive singular sets for 3-flows are partially hyperbolic
  attractors or repellers}.
\newblock {\em {Ann. of Math. (2)}}, {160}({2}):{375--432}, {2004}.

\bibitem{MSY} K. Moriyasu, K. Sakai and K. Yamamoto, \textit{Regular maps with the specification
property}, Discrete and Cointinuous Dynam. Sys., {\bf 33} (2013), 2991--3009.


\bibitem{PT} S.Yu.Pilyugin and S.B.Tikhomirov,\textit{Lipschitz shadowing implies structural stability}, Nonlinearity, {\bf 23} (2010), 2509--2515. 


\bibitem{OT} 
K. Oliveira and X.Tian, 
\textit{Non-uniform hyperbolicity and non-uniform specification}, Trans. Amer. Math. Soc., 
{\bf 365} (2013), 4371--4392.


\bibitem{Palis}
J. Palis
\newblock A global view of dynamics and a conjecture on the denseness of finitude of attractors.
\newblock G\'eom\'etrie complexe et systems dynamiques (Orsay, 1995). Ast\'erisque No. 261 (2000), 
335--347. 

\bibitem{PdM} J.~Palis, W. de Melo,
\newblock Geometric Theory of Dynamical Systems: An Introduction.
\newblock Springer Verlag, 1982.

\bibitem{PS05}
C.-E.~Pfister and W. Sullivan, \textit{Large Deviations Estimates for Dynamical Systems without the
Specification Property Application to the Beta-Shifts}, Nonlinearity, {\bf 18}
(2005), 237--261.

\bibitem{S} K.~Sakai, \textit{Pseudo-orbit tracing property and 
strong transversality of diffeomorphisms on closed manifolds}, Osaka J. Math. 
31 (1994), no. 2, 373--386.

\bibitem{SSY} K.~Sakai, N.~Sumi and K.~Yamamoto, \textit{Diffeomorphisms satisfying the specification property},
Proc. Amer. Math. Soc., {\bf 138} (2009), 315--321.

\bibitem{Sal} L. Salgado, {\it Partially dominated splittings}, Preprint ArXiv:1402.1511.


\bibitem{Shub}
M. Shub, Global stability of dynamical systems. With the collaboration of Albert Fathi and R\'emi Langevin. Translated from the French by Joseph Christy. Springer-Verlag, New York, 1987. xii+150 pp.

\bibitem{SVY} N.~Sumi, P. Varandas and K.~Yamamoto, Partial hyperbolicity and specification, 
Proc. Amer. Math. Soc. (to appear)



\bibitem{TaV00}
F.~Takens and E.~Verbitskiy, \textit{Multifractal analysis of local entropies for expansive homeomorphisms with specification}, Comm. Math. Phys. {\bf 203} (1999), 593--612.

\bibitem{TaV03}
F.~Takens and E.~Verbitskiy. \textit{On the variational principle for the topological entropy of certain noncompact sets}, Ergod. Th. \& Dynam. Sys {\bf 23} (2003), 317--348.


\bibitem{Th10}
D.~Thompson, \textit{Irregular sets, the $\beta$-transformation and the almost specification property}, Trans. Amer. Math, Soc. {\bf 364} (2012), 5395--5414.

\bibitem{Va12}
P.~Varandas, \textit{Non-uniform specification and large deviations for weak Gibbs measures}, J. Statist. Phys., {\bf 146} (2012), 330--358.


\bibitem{Y}
L.~S. Young, \textit{Some large deviation results for dynamical systems}, Trans. Am. Math. Soc. {\bf 318} (1990) 525--543.


\bibitem{Vi03}
T.~Vivier, \textit{Flots robustement transitifs sur les vari\'et\'es compactes}, C. R. Math. Acad. Sci. Paris, {\bf 337} (2003), 791--796.

\bibitem{Wen}
L. Wen, On the $C^1$ Stability Conjecture for Flows,
Journal of Differential Equations, 129: 2, 334--357, 1996.


\end{thebibliography}
\end{document}